\documentclass[11pt,notitlepage]{article}
\usepackage{amssymb,amsthm,bbm,graphicx,placeins,enumitem}
\usepackage{amstext}
\usepackage{natbib}
\usepackage{setspace,mathtools,thmtools,thm-restate}
\usepackage[left=1.25in,right=1.25in,top=1in,bottom=1in]{geometry} 
\usepackage{float}
\usepackage{color,colortbl}
\usepackage{changepage}
\usepackage{booktabs}
\usepackage{array}
\usepackage{multirow}
\usepackage[font=small]{caption}
\usepackage{subcaption}
\usepackage[ruled,norelsize]{algorithm2e}
\usepackage{mathrsfs}  
\usepackage{thm-restate}
\usepackage{csquotes}
\usepackage[noblocks]{authblk}
\usepackage{titlesec}
\titleformat*{\section}{\Large\bfseries}
\usepackage[ruled,norelsize]{algorithm2e}

\usepackage{hyperref}
\hypersetup{colorlinks,
linkcolor=blue,
filecolor=blue,
urlcolor=blue,
citecolor=blue}
\allowdisplaybreaks

\DeclareMathOperator*{\argmax}{arg\,max}

\newcommand{\indicate}[1]{\mathbf{1}_{\{#1\}}}

\newtheorem{theorem}{Theorem}
\newtheorem{proposition}{Proposition}

\newtheorem{assumption}{Assumption}

\theoremstyle{definition}
\newtheorem{example}{Example}

\DeclareCaptionSubType[Alph]{figure}
\newcommand*{\Atilde}{\skew{5}{\tilde}{\mathcal A}}

\def\pid{\pi^\textnormal{d}}
\def\pis{\pi^\textnormal{s}}
\def\pia{\pi^\textnormal{a}}

\def\E{\mathbf{E}}
\def\P{\mathbf{P}}

\def\x{\mathbf{x}}
\def\w{\mathbf{w}}
\def\W{\mathbf{W}}
\def\a{\mathbf{a}}

\def\y{\mathbf{y}}
\def\taue{{\tau_{\text{e}}}}
\def\taus{{\tau_{\text{s}}}}

\newlength\mylen

\titleformat*{\section}{\large\bfseries}
\titleformat*{\subsection}{\normalsize\bfseries}
\titleformat*{\subsubsection}{\normalsize\bfseries}
\titleformat*{\paragraph}{\normalsize\bfseries}
\titleformat*{\subparagraph}{\normalsize\bfseries}
\titlespacing*{\section}{0pt}{9pt}{9pt}
\titlespacing*{\subsection}{0pt}{7pt}{7pt}
\titlespacing*{\subsubsection}{0pt}{7pt}{7pt}

\begin{document}

\onehalfspacing

\title{\textbf{\Large Monte Carlo Tree Search with Sampled\\
  Information Relaxation Dual Bounds}}
\author{%
Daniel R. Jiang, Lina Al-Kanj, Warren B. Powell
}
\maketitle
\vspace{0em}
\begin{abstract}
Monte Carlo Tree Search (MCTS), most famously used in game-play artificial intelligence (e.g., the game of Go), is a well-known strategy for constructing approximate solutions to sequential decision problems. Its primary innovation is the use of a heuristic, known as a \emph{default policy}, to obtain Monte Carlo estimates of downstream values for states in a decision tree. This information is used to iteratively expand the tree towards regions of states and actions that an optimal policy might visit. However, to guarantee convergence to the optimal action, MCTS requires the entire tree to be expanded asymptotically. In this paper, we propose a new technique called Primal-Dual MCTS that utilizes sampled \emph{information relaxation} upper bounds on potential actions, creating the possibility of ``ignoring'' parts of the tree that stem from highly suboptimal choices. This allows us to prove that despite converging to a partial decision tree in the limit, the recommended action from Primal-Dual MCTS is optimal. The new approach shows significant promise when used to optimize the behavior of a single driver navigating a graph while operating on a ride-sharing platform. Numerical experiments on a real dataset of 7{,}000 trips in New Jersey suggest that Primal-Dual MCTS improves upon standard MCTS by producing deeper decision trees and exhibits a reduced sensitivity to the size of the action space.

\end{abstract}

\setlist[itemize]{topsep=2pt,itemsep=0pt}
\setlength{\abovedisplayskip}{0.2cm}
\setlength{\belowdisplayskip}{0.2cm}
\setlist[itemize]{topsep=1pt,itemsep=-1ex}
\setlist[enumerate]{topsep=1pt,itemsep=-1ex}
\setlength{\belowcaptionskip}{0pt}

\vspace{20pt}
\section{Introduction}
The Monte Carlo Tree Search (MCTS) is a technique popularized by the artificial intelligence (AI) community \citep{Coulom2007} for solving sequential decision problems with finite state and action spaces.
To avoid searching through an intractably large decision tree, MCTS instead iteratively builds the tree and attempts to focus on regions composed of states and actions that an optimal policy might visit. A heuristic known as the \emph{default policy} is used to provide Monte Carlo estimates of downstream values, which serve as a guide for MCTS to explore promising regions of the search space. When the allotted computational resources have been expended, the hope is that the best first stage decision recommended by the \emph{partial decision tree} is a reasonably good estimate of the optimal decision that would have been implied by the full tree.

The applications of MCTS are broad and varied, but the strategy is traditionally most often applied to game-play AI \citep{Chaslot2008}. To name a few specific applications, these include Go \citep{Chaslot2006,Gelly2011,Gelly2012,Silver2016}, Othello \citep{Hingston2007,Nijssen2007,Osaki2008,Robles2011}, Backgammon \citep{VanLishout2007}, Poker \citep{Maitrepierre2008,VandenBroeck2009,Ponsen2010}, 16 $\times$ 16 Sudoku \citep{Cazenave2009}, and even general game playing AI \citep{Mehat2010}. We remark that a characteristic of games is that the transitions from state to state are deterministic; because of this, the standard specification for MCTS deals with deterministic problems. The ``Monte Carlo'' descriptor in name of MCTS therefore refers to stochasticity in the default policy. A particularly thorough review of both the MCTS methodology and its applications can be found in \cite{Browne2012a}.

The adaptive sampling algorithm by \cite{Chang2005}, introduced within the operations research (OR) community, leverages a well-known bandit algorithm called UCB (upper confidence bound) for solving MDPs. The UCB approach is also extensively used for successful implementations of MCTS \citep{Kocsis2006}. Although the two techniques share similar ideas, the OR community has generally not taken advantage of the MCTS methodology in applications, with the exception of two recent papers. The paper \cite{Bertsimas2014} compares MCTS with rolling horizon mathematical optimization techniques (a standard method in OR) on a large scale dynamic resource allocation problem, specifically that of tactical wildfire management. \cite{Al-Kanj2016} applies MCTS to an information-collecting vehicle routing problem, which is an extension of the classical vehicle routing model where the decisions now depend on a \emph{belief state}. Not surprisingly, both of these problems are intractable via standard Markov decision process (MDP) techniques, and results from these papers suggest that MCTS could be a viable alternative to other approximation methods (e.g., approximate dynamic programming). However, \cite{Bertsimas2014} finds that MCTS is competitive with rolling horizon techniques only on smaller instances of the problems and their evidence suggests that MCTS can be quite sensitive to large action spaces. In addition, they observe that large actions spaces are more detrimental to MCTS than large state spaces. These observations form the basis of our \emph{first research motivation}: can we control the action branching factor by making ``intelligent guesses'' at which actions may be suboptimal? If so, potentially suboptimal actions can be ignored.

Next, let us briefly review the currently available convergence theory. The work of \cite{Kocsis2006} uses the UCB algorithm to sample actions in MCTS, resulting in an algorithm called UCT (upper confidence trees). A key property of UCB is that every action is sampled infinitely often and \cite{Kocsis2006} exploit this to show that the probability of selecting a suboptimal action converges to zero at the root of the tree. \cite{Silver2010} use the UCT result as a basis for showing convergence of a variant of MCTS for partially observed MDPs. \cite{Coutoux2011} extends MCTS for deterministic, finite state problems to stochastic problems with continuous state spaces using a technique called \emph{double progressive widening}. The paper \cite{Auger2013} provides convergence results for MCTS with double progressive widening under an action sampling assumption. In these papers, the asymptotic convergence of MCTS relies on some form of ``exploring every node infinitely often.'' However, given that the spirit of the algorithm is to build partial trees that are biased towards nearly optimal actions, we believe that an alternative line of thinking deserves further study. Thus, our \emph{second research motivation} is: can we design a version of MCTS that asymptotically \emph{does not} expand the entire tree, yet is still optimal?

By far, the most significant recent development in this area is Google Deepmind's development of AlphaGo, the first computer to defeat a human player in the game of Go, of which MCTS plays a major role \citep{Silver2016}. The authors state, ``The strongest current Go programs are based on MCTS, enhanced by policies that are trained to predict human expert moves.'' To be more precise, the \emph{default policy} used by AlphaGo is carefully constructed through several steps: (1) a classifier to predict expert moves is trained using 29.4 million game positions from 160,000 games on top of a deep convolutional neural network (consisting of 13 layers); (2) the classifier is then played against itself and a policy gradient technique is used to develop a policy that aims to win the game rather than simply mimic human players; (3) another deep convolutional neural network is used to approximate the \emph{value function} of the heuristic policy; and (4) a combination of the two neural networks, dubbed the policy and value networks, provides an MCTS algorithm with the default policy and the estimated downstream values. The current discussion is a perfect illustration of the motivation behind our \emph{third research motivation}: if such a remarkable amount of effort is used to design a default policy, can we develop techniques to further exploit this heuristic guidance within the MCTS framework?

In this paper, we address each of these questions by proposing a novel MCTS method, called \emph{Primal-Dual MCTS} (the name is inspired by \cite{Andersen2004}), that takes advantage of the \emph{information relaxation bound} idea (also known as \emph{martingale duality}) first developed in \cite{Haugh2004} and later generalized by \cite{Brown2010}. The essence of information relaxation is to relax nonanticipativity constraints (i.e., allow the decision maker to use future information) in order to produce upper bounds on the objective value (assuming a maximization problem). To account for the issue that a naive use of future information can produce weak bounds, \cite{Brown2010} describes a method to penalize the use of future information so that one may obtain a tighter (smaller) upper bound. This is called a \emph{dual approach} and it is shown that value of the upper bound can be made equal to the optimal value if a particular penalty function is chosen that depends on the optimal value function of the original problem. Information relaxation has been used successfully to estimate the sub-optimality of policies in a number of application domains, including option pricing \citep{Andersen2004}, portfolio optimization \citep{Brown2011}, valuation of natural gas \citep{Lai2011b,Nadarajah2015}, optimal stopping \citep{Desai2012a}, and vehicle routing \citep{Goodson2016}. 
More specifically, the contributions of this paper are as follows.
\begin{itemize}
\item We propose a new MCTS method called Primal-Dual MCTS that utilizes the information relaxation methodology of \cite{Brown2010} to generate dual upper bounds. These bounds are used when MCTS needs to choose actions to explore (this is known as \emph{expansion} in the literature). When the algorithm considers performing an expansion step, we obtain \emph{sampled} upper bounds (i.e., in expectation, they are greater than the optimal value) for a set of potential actions and select an action with an upper bound that is \emph{better} than the value of the current optimal action. Correspondingly, if all remaining unexplored actions have upper bounds lower than the value of the current optimal action, then we do not expand further. This addresses our first research motivation of reducing the branching factor in a principled way.

\item We prove that our method converges to the optimal action (and optimal value) at the root node. This holds even though our proposed technique does not preclude the possibility of a partially expanded tree in the limit. By carefully utilizing the upper bounds, we are able to ``provably ignore'' entire subtrees, thereby reducing the amount of computation needed. This addresses our second research motivation, which extends the current convergence theory of MCTS.

\item Although there are many ways to construct the dual bound, one special instance of Primal-Dual MCTS uses the default policy (the heuristic for estimating downstream values) to induce a penalty function. This addresses our third research motivation: the default policy can provide actionable information in the form of upper bounds, in addition to the original intention of estimating downstream values.

\item Lastly, we present a model of the stochastic optimization problem faced by a single driver who provides transportation for fare-paying customers while navigating a graph. The problem is motivated by the need for ride-sharing platforms (e.g., Uber and Lyft) to be able to accurately simulate the operations of an entire ride-sharing system/fleet. Understanding human drivers' behaviors is crucial to a smooth integration of platform controlled driver-less vehicles with the traditional contractor model (e.g., in Pittsburgh, Pennsylvania). Our computational results show that Primal-Dual MCTS dramatically reduces the breadth of the search tree when compared to standard MCTS.
\end{itemize}

The paper is organized as follows. In Section \ref{sec:prelim}, we describe a general model of a stochastic sequential decision problem and review the standard MCTS framework along with the duality and information relaxation procedures of \cite{Brown2010}. We present the algorithm, Primal-Dual MCTS, in Section \ref{sec:alg}, and provide the convergence analysis in Section \ref{sec:conv}. The ride-sharing model and the associated numerical results are discussed in Section \ref{sec:num} and we provide concluding remarks in Section \ref{sec:conc}.

\section{Preliminaries}
\label{sec:prelim}
In this section, we first formulate the mathematical model of the underlying optimization problem as an MDP. Because we are in the setting of decision trees and information relaxations, we need to extend traditional MDP notation with some additional elements. We also introduce the existing concepts, methodologies, and relevant results that are used throughout the paper. 

\subsection{Mathematical Model}
\label{subsec:model}
As is common in MCTS, we consider an underlying MDP formulation with a finite horizon $t = 0,1,\ldots,T$ where the set of decision epochs is $\mathcal T = \{0,1,\ldots, T-1\}$. Let $\mathcal S$ be a \emph{state space} and $\mathcal A$ be an \emph{action space} and we assume a finite state and action setting: $|\mathcal S| < \infty$ and $|\mathcal A| < \infty$. The set of feasible actions for state $s\in \mathcal S$ is $\mathcal A_s$, a subset of $\mathcal A$. The set $\mathcal U = \{(s,a) \in \mathcal S \times \mathcal A : a \in \mathcal A_s\}$ contains all feasible state-action pairs.

The dynamics from one state to the next depend on the action taken at time $t$, written $a_t \in \mathcal A$, and an exogenous (i.e., independent of states and actions) random process $\{W_t\}_{t=1}^T$ on $(\Omega, \mathcal F, \P)$ taking values in a finite space $\mathcal W$. For simplicity, we assume that $W_{t}$ are independent across time $t$. The \emph{transition function} is given by $f: \mathcal S \times \mathcal A \times \mathcal W \rightarrow \mathcal S$. We denote the deterministic initial state by $s_0 \in \mathcal S$ and let $\{S_t\}_{t=0}^T$ be the random process describing the evolution of the system state, where $S_0 = s_0$ and $S_{t+1} = f(S_t, a_t, W_{t+1})$. To distinguish from the random variable $S_t \in \mathcal S$, we shall refer to a particular element of the state space by lowercase variables, e.g., $s \in \mathcal S$. The contribution (or reward) function at stage $t < T$ is given by $c_t: \mathcal S \times \mathcal A \times \mathcal W \rightarrow \mathbb R$. For a fixed state-action pair $(s,a) \in \mathcal U$, the contribution is the random quantity $c_t(s,a,W_{t+1})$, which we assume is bounded.

Because there are a number of other ``policies'' that the MCTS algorithm takes as input parameters (to be discussed in Section \ref{subsec:mcts}), we call the main MDP policy of interest the \emph{operating policy}. Let $\Pi$ be the set of all policies for the MDP with a generic element $\pi = \{ \pi_0, \pi_2, \ldots, \pi_{T-1} \} \in \Pi$. Each decision function $\pi_t : \mathcal S \rightarrow \mathcal A$ is a deterministic map from the state space to the action space, such that $\pi_t(s) \in \mathcal A_s$ for any state $s \in \mathcal S$. Finally, we define the objective function, which is to maximize expected cumulative contribution over the finite time horizon:
\begin{equation}
\max_{\; \, \pi \in \Pi} \; \E \left[  \sum_{t=0}^{T-1} c_t (S_t,\pi_t(S_t), W_{t+1})  \, \bigl | \, S_0 = s_0 \right].
\label{eq:mdpobj}
\end{equation}
Let $V_t^*(s)$ be the optimal value function at state $s$ and time $t$. It can be defined via the standard Bellman optimality recursion:
\begin{align*}
&V_t^*(s) = \max_{a \in \mathcal A_s} \; \E  \left[  \, c_t(s,a,W_{t+1})   + V_{t+1}^*(S_{t+1}) \right] \; \text{for all } s \in \mathcal S,\; t \in \mathcal T,\\
&V_T^*(s) = 0 \; \text{for all } s \in \mathcal S.
\end{align*}
The state-action formulation of the Bellman recursion is also necessary for the purposes of MCTS as the decision tree contains both state and state-action nodes. The state-action value function is defined as:
\begin{align*}
&Q_t^*(s,a) =  \E  \left[  \, c_t(s,a,W_{t+1})   + V_{t+1}^*(S_{t+1}) \right] \; \text{for all } (s,a) \in \mathcal U,\; t \in \mathcal T.
\end{align*}
For consistency, it is also useful to let $Q_T^*(s,a) = 0$ for all $(s,a)$. It thus follows that $V_t^*(s) = \max_{a \in \mathcal A_s} Q_t^*(s,a)$. Likewise, the optimal policy $\pi^{*} = \{\pi_0^*,\ldots,\pi_{T-1}^*\}$ from the set $\Pi$ is characterized by $\pi^*_t(s) = \argmax_{a \in \mathcal A_s} Q_t^*(s,a)$.

It is also useful for us to define the value of a particular operating policy $\pi$ starting from a state $s\in \mathcal S$ at time $t$, given by the value function $V_t^\pi(s)$. If we let $S_{t+1}^\pi = f(s,\pi_t(s),W_{t+1})$, then the following recursion holds:
\begin{equation}
\begin{aligned}
&V_t^\pi(s) = \E  \left[  \, c_t(s,\pi_t(s),W_{t+1})   + V_{t+1}^\pi(S_{t+1}^\pi) \right] \; \text{for all } s \in \mathcal S,\; t \in \mathcal T,\\
&V_T^\pi(s) = 0 \; \text{for all } s \in \mathcal S.
\end{aligned}
\label{eq:bellman1}
\end{equation}
Similarly, we have
\begin{equation}
\begin{aligned}
&Q_t^\pi(s,a) =  \E  \left[  \, c_t(s,a,W_{t+1})   + V_{t+1}^\pi(S_{t+1}^\pi) \right] \; \text{for all } (s,a) \in \mathcal S \times \mathcal A,\; t \in \mathcal T,\\
&Q_T^\pi(s,a) = 0 \; \text{for all } (s,a) \in \mathcal S \times \mathcal A,
\end{aligned}
\label{eq:bellman2}
\end{equation}
the state-action value functions for a given operating policy $\pi$. 

Suppose we are at a fixed time $t$. Due to the notational needs of information relaxation, let $s_{\tau,t}(s,\a,\w) \in \mathcal S$ be the deterministic state reached at time $\tau > t$ given that we are in state $s$ at time $t$, implement a fixed sequence of actions $\a = (a_t,a_{t+1}, \ldots a_{T-1})$, and observe a fixed sequence of exogenous outcomes $\w = (w_{t+1}, w_{t+2}, \ldots, w_T)$. For succinctness, the time subscripts have been dropped from the vector representations. Similarly, let $s_{\tau,t}(s,\pi,\w) \in \mathcal S$ be the deterministic state reached at time $\tau > t$ if we follow a fixed policy $\pi \in \Pi$.

Finally, we need to refer to the future contributions starting from time $t$, state $s$, and a sequence of exogenous outcomes $\w = (w_{t+1}, w_{t+2}, \ldots, w_T)$. For convenience, we slightly abuse notation and use two versions of this quantity, one using a fixed sequence of actions $\a = (a_t, a_{t+1}, \ldots, a_{T-1})$ and another using a fixed policy $\pi$:
\[
h_t(s,\a,\w)  = \sum_{\tau = t}^{T-1} c_\tau(s_{\tau,t}(s,\a,\w), a_\tau, w_{\tau+1}), \; \; h_t(s,\pi,\w)  = \sum_{\tau = t}^{T-1} c_\tau(s_{\tau,t}(s,\pi,\w), a_\tau, w_{\tau+1}).
\]
Therefore, if we define the random process $\W_{t+1,T} = (W_{t+1}, W_{t+2}, \ldots, W_T)$, then the quantities $h_t(s,\a,\W_{t+1,T})$ and $h_t(s,\pi,\W_{t+1,T})$ represent the \emph{random downstream cumulative reward} starting at state $s$ and time $t$, following a deterministic sequence of actions $\a$ or a policy $\pi$. For example, the objective function to the MDP given in (\ref{eq:mdpobj}) can be rewritten more concisely as $\max_{ \pi \in \Pi} \E \bigl[  h_t (s_0,\pi, \W_{1,T}) \bigr]$.

\subsection{Monte Carlo Tree Search}
\label{subsec:mcts}
The canonical MCTS algorithm iteratively grows and updates a decision tree, using the default policy as a guide towards promising subtrees. Because sequential systems evolve from a (pre-decision) state $S_t$, to an action $a_t$, to a post-decision state or a state-action pair $(S_t,a_t)$, to new information $W_{t+1}$, and finally, to another state $S_{t+1}$, there are two types of nodes in a decision tree: \emph{state nodes} (or ``pre-decision states'') and \emph{state-action nodes} (or ``post-decision states''). The layers of the tree are chronological and alternate between these two types of nodes. A child of a state node is a state-action node connected by an edge that represents a particular action. Similarly, a child of a state-action node is a state node for the next stage, where the edge represents an outcome of the exogenous information process $W_{t+1}$. 

Since we are working within the decision tree setting, it is necessary to introduce some additional notation that departs from the traditional MDP style. A state node is represented by an augmented state that contains the entire path down the tree from the root node $s_0$:
\[
\mathbf{x}_t = (s_0, a_0,s_1,a_1,s_2\ldots,a_{t-1},s_t) \in \mathcal X_t,
\]
where $a_0 \in \mathcal A_{s_0}, a_1 \in \mathcal A_{s_1}, \ldots, a_{t-1} \in \mathcal A_{s_{t-1}}$ and $s_1, s_2, \ldots, s_t \in \mathcal S$. Let $\mathcal X_t$ be the set of all possible $\x_t$ (representing all possible paths to states at time $t$).
A state-action node is represented via the notation $\mathbf{y}_t = (\mathbf{x}_t,a_t)$ where $a_t \in \mathcal A_{s_t}$.  Similarly, let $\mathcal Y_t$ be the set of all possible $\y_t$. We can take advantage of the Markovian property along with the fact that any node $\x_t$ or $\y_t$ contains information about $t$ to write (again, a slight abuse of notation)
\[
V^*(\x_t) = V_t^*(s_t) \quad \text{and} \quad Q^*(\mathbf{y}_t) = Q_t^*(\x_t,a_t) = Q_t^*(s_t,a_t).
\]
At iteration $n$ of MCTS, each state node $\x_t$ is associated with a value function approximation $\bar{V}^n(\x_t)$ and each state-action node $(\x_t,a_t)$ is associated with the state-action value function approximation $\bar{Q}^n(\x_t,a_t)$. Moreover, we use the following shorthand notation:
\[
\P(S_{t+1} = s_{t+1} \, | \, \y_t) = \P(S_{t+1} = s_{t+1} \, | \, \x_t,a_t) = \P(f(s_t,a_t,W_{t+1}) = s_{t+1}).
\]

There are four main phases in the MCTS algorithm: selection, expansion, simulation, and backpropagation \citep{Browne2012a}. Oftentimes, the first two phases are called the \emph{tree policy} because it traverses and expands the tree; it is in these two phases where we will introduce our new methodology. Let us now summarize the steps of MCTS while employing \emph{double progressive widening} (DPW) technique \citep{Coutoux2011} to control the branching at each level of the tree. As its name suggests, DPW means we slowly expand the branching factor of the tree, in both state nodes and state-action nodes. The following steps summarize the steps of MCTS at a particular iteration $n$.
\begin{itemize}
	\item[] \textbf{Selection.} We are given a \emph{selection policy}, which determines a path down the tree at each iteration. When no progressive widening is needed, the algorithm traverses the tree until it reaches a leaf node, i.e., an unexpanded state node, and proceeds to the simulation step. On the other hand, when progressive widening is needed, the traversal is performed until an \emph{expandable node}, i.e., one for which there exists a child that has not yet been added to the tree, is reached. This could be either a state node or a state-action node; the algorithm now proceeds to the expansion step.
	\item[] \textbf{Expansion.} We now utilize a given \emph{expansion policy} to decide which child to add to the tree. The simplest method, of course, is to add an action at random or add an exogenous state transition at random. Assuming that expansion of a state-action node always follows the expansion of a state node, we are now in a leaf state node.
	\item[] \textbf{Simulation.} The aforementioned \emph{default policy} is now used to generate a sample of the value function evaluated at the current state node. The estimate is constructed using a sample path of the exogenous information process. This step of MCTS is also called a \emph{rollout}.
	\item[] \textbf{Backpropagation.} The last step is to recursively update the values up the tree until the root node is reached: for state-action nodes, a weighted average is performed on the values of its child nodes to update $\bar{Q}_t^n(\x_t,a_t)$, and for state nodes, a combination of a weighted average and maximum of the values of its child nodes is taken to update $\bar{V}_t^n(\x_t)$. These operations correspond to a backup operator discussed in \cite{Coulom2007} that achieves good empirical performance. We now move on to the next iteration by starting once again with the selection step.
\end{itemize}
Once a pre-specified number of iterations have been run, the best action out of the root node is chosen for implementation. After landing in a new state in the real system, MCTS can be run again with the new state as the root node. A practical strategy is to use the relevant subtree from the previous run of MCTS to initialize the new process \citep{Bertsimas2014}.

\subsection{Information Relaxation Bounds}
We next review the information relaxation duality ideas from \cite{Brown2010}; see also \cite{Brown2011} and \cite{Brown2014}. Here, we adapt the results of \cite{Brown2010} to our setting, where we require the bounds to hold for arbitrary sub-problems of the MDP. Specifically, we state the theorems from the point of view of a specific time $t$ and initial state-action pair $(s,a)$. Also, we focus on the \emph{perfect information relaxation}, where one assumes full knowledge of the future in order to create upper bounds. In this case, we have
\[
V_t^*(s) \le \E \Bigl[ \max_{\a}  h_t(s,\a,\W_{1,T}) \Bigr],
\]
which means that the value achieved by the optimal policy starting from time $t$ is upper bounded by the value of the policy that selects actions using perfect information. As we described previously, the main idea of this approach is to relax nonanticipativity constraints to provide upper bounds. Because these bounds may be quite weak, they are subsequently strengthened by imposing penalties for usage of future information. To be more precisely, we would like to subtract away a penalty defined by a function $z_t$ so that the right-hand-side is decreased to: $\E \left[ \max_{\a} \bigl[  h_t(s,\a,\W_{t+1,T}) - z_t(s,\a,\W_{t+1,T}) \bigr] \right]$.

Consider the subproblem (or subtree) starting in stage $t$ and state $s$. A \emph{dual penalty} $z_{t}$ is a function that maps an initial state, a sequence of actions $\a = (a_t,a_{t+1},\ldots,a_{T-1})$, and a sequence of exogenous outcomes $\w = (w_{t+1}, \ldots, w_{T})$ to a penalty $z_{t}(s,\a, \w) \in \mathbb R$. As we did in the definition of $h_t$, the same quantity is written $z_{t}(s,\pi,\w)$ when the sequence of actions is generated by a policy $\pi$. The set of \emph{dual feasible penalties} for a given initial state $s$ are those $z_t$ that do not penalize admissible policies; it is given by the set
\begin{equation}
\mathcal Z_t(s) = \bigl \{ z_t :   \mathbf{E} \bigl[  z_{t}(s, \pi, \W_{t+1,T} )\bigr] \le 0 \; \; \forall \, \pi \in \Pi  \bigr\},
\label{eq:dualdef}
\end{equation}
where $\W_{t+1,T} = (W_{t+1}, \ldots, W_{T})$.
Therefore, the only ``primal'' policies (i.e., policies for the original MDP) for which a dual feasible penalty $z$ could assign positive penalty in expectation are those that are not in $\Pi$. 

We now state a theorem from \cite{Brown2010} that illuminates the dual bound method. The intuition is best described from a simulation point of view: we sample an entire future trajectory of the exogenous information $\W_{t+1,T}$ and using full knowledge of this information, the optimal actions are computed. It is clear that after taking the average of many such trajectories, the corresponding averaged objective value will be an upper bound on the value of the optimal (nonanticipative) policy. The dual penalty is simply a way to improve this upper bound by \emph{penalizing the use of future information}; the only property required in the proof of Theorem \ref{thm:weak} is the definition of dual feasibility. The proof is simple and we repeat it here so that we can state a small extension later in the paper (in Proposition \ref{prop:upperbound}). The right-hand-side of the inequality below is a penalized perfect information relaxation.

\begin{theorem}[Weak Duality, \cite{Brown2010}]
Fix a stage $t \in \mathcal T$ and initial state $s \in \mathcal S$. Let $\pi \in \Pi$ be a feasible policy and $z_t \in \mathcal Z_t(s)$ be a dual feasible penalty, as defined in (\ref{eq:dualdef}). It holds that
\begin{equation}
V_t^\pi(s) \le \E \left[ \max_{\a} \bigl[  h_t(s,\a,\W_{t+1,T}) - z_t(s,\a,\W_{t+1,T}) \bigr] \right],
\label{eq:weak}
\end{equation}
where $\a = (a_t, \ldots, a_{T-1})$.
\label{thm:weak}
\end{theorem}

\begin{proof}
By definition, $V_t^\pi(s) = \E \bigl[  h_t(s,\pi,\W_{t+1,T}) \bigr]$. Thus, it follows by dual feasibility that
\begin{align*}
V_t^\pi(s_t) &\le \E \bigl[  h_t(s,\pi,\W_{t+1,T}) - z_t(s,\pi,\W_{t+1,T})  \bigr]\\
&\le \E \Bigl[ \max_{\a} \bigl[  h_t(s,\a,\W_{t+1,T})  - z_t(s,\a,\W_{t+1,T}) \bigr] \Bigr].
\end{align*}
The second inequality follows by the property that a policy using future information must achieve higher value than an admissible policy. In other words, $\Pi$ is contained within the set of policies that are not constrained by nonanticipativity.
\end{proof}

Note that the left-hand-side of (\ref{eq:weak}) is known as the \emph{primal problem} and the right-hand-side is the \emph{dual problem}, so it is easy to see that the theorem is analogous to classical duality results from linear programming. The next step, of course, is to identify some dual feasible penalties. For each $t$, let $\nu_t: \mathcal S \rightarrow \mathbb R$ be any function and define
\begin{equation}
\bar{\nu}_\tau(s,\a,\w) = \nu_{\tau+1}(s_{\tau+1}(t,s,\a,\w)) - \E \,  \nu_{\tau+1} (f(s_\tau(t,s,\a,\w),a_\tau,W_{\tau+1})).
\label{eq:vfa}
\end{equation}
 \cite{Brown2010} suggests the following additive form for a dual penalty:
\begin{equation}
z_t^\nu(s,\a,\w) = \sum_{\tau = t}^{T-1} \bar{\nu}_{\tau}(s,\a,\w),
\label{eq:vfa2}
\end{equation}
 and it is shown in the paper that this form is indeed dual feasible. We refer to this as the \emph{dual penalty generated by $\nu = \{\nu_t\}$}. The \emph{standard dual upper bound} is accomplished without penalizing, i.e., by setting $\nu_t \equiv 0$ for all $t$. As we will show in our empirical results on the ride-sharing model, this upper bound is simple to implement and may be quite effective. 

 However, in situations where the standard dual upper bound is too weak, a good choice of $\nu$ can generate tighter bounds. It is shown that if the optimal value function $V_\tau^*$ is used in place of $\nu_\tau$ in (\ref{eq:vfa}), then the best upper bound is obtained. In particular, a form of \emph{strong duality} holds: when Theorem \ref{thm:weak} is invoked using the optimal policy $\pi^* \in \Pi$ and $\nu_\tau = V_\tau^*$, the inequality (\ref{eq:weak}) is achieved with equality. The interpretation of $\nu_\tau = V_\tau^*$ is that $d_\tau^\nu$ can be thought of informally as the ``value gained from knowing the future.'' Thus, the intuition behind this result is as follows: if one knows precisely how much can be gained by using future information, then a perfect penalty can be constructed so as to recover the optimal value of the primal problem.

However, strong duality is hard to exploit in practical settings, given that both sides of the equation require knowledge of the optimal policy. Instead, a viable strategy is to use approximate value functions $\bar{V}_\tau$ on the right-hand-side of (\ref{eq:vfa}) in order to obtain ``good'' upper bounds on the optimal value function $V_t^*$ on the left-hand-side of (\ref{eq:weak}).   This is where we can potentially take advantage of the default policy of MCTS to improve upon the standard dual upper bound; the value function associated with this policy can be used to generate a dual feasible penalty. We now state a specialization of Theorem \ref{thm:weak} that is useful for our MCTS setting.


%

\begin{proposition}[State-Action Duality]
Fix a stage $t \in \mathcal T$ and an initial state-action pair $(s,a) \in \mathcal S \times \mathcal A$. Assume that the dual penalty function takes the form given in (\ref{eq:vfa})--(\ref{eq:vfa2}) for some $\nu = \{\nu_t\}$. Then, it holds that
\begin{equation}
Q_t^*(s,a) \le \E \Bigl[ \, c_t(s,a,W_{t+1}) +  \max_{\a} \bigl[ h_{t+1}(S_{t+1},\a,\W_{t+1,T}) - z^\nu_{t+1}(S_{t+1},\a,\W_{t+1,T}) \bigr] \Bigr],
\label{eq:upperbound}
\end{equation}
where $S_{t+1} = f(s,a,W_{t+1})$ and the optimization is over the vector $\a = (a_{t+1},\ldots, a_{T-1})$.
\label{prop:upperbound}
\end{proposition}
\begin{proof}
Choose a policy $\tilde{\pi}$ (restricted to stage $t$ onwards) such that the first decision function maps to $a$ and the remaining decision functions match those of the optimal policy $\pi^*$: 
\[
\tilde{\pi} = (a, \pi_{t+1}^*, \pi_{t+2}^*, \ldots, \pi_{T-1}^*).
\]
Using this policy and the separability of $z^\nu_t$ given in (\ref{eq:vfa2}), an argument analogous to the proof of Theorem \ref{thm:weak} can be used to obtain the result.
\end{proof}
For convenience, let us denote the dual upper bound generated using the functions $\nu$ by
\begin{equation*}
u_t^\nu(s,a) = \E \Bigl[ \, c_t(s,a,W_{t+1}) +  \max_{\a} \bigl[ h_{t+1}(S_{t+1},\a,\W_{t+1,T}) - z_{t+1}^\nu(S_{t+1},\a,\W_{t+1,T}) \bigr] \Bigr].
\end{equation*}
Therefore, the dual bound can be simply stated as $Q_t^*(s,a) \le u_t^\nu(s,a)$. For a state-action node $\y_t = (s_0,a_0,\ldots,s_t,a_t)$ in the decision tree, we use the notation $u^\nu(\y_t) = u_t^\nu(s,a)$. The proposed algorithm will keep estimates of the upper bound on the right-hand-side of (\ref{eq:upperbound}) in order to make tree expansion decisions. As the algorithm progresses, the estimates of the upper bound are refined using a stochastic gradient method.


\section{Primal-Dual MCTS Algorithm}
\label{sec:alg}
In this section, we formally describe the proposed Primal-Dual MCTS algorithm. The core of the algorithm is MCTS with double progressive widening \citep{Coutoux2011}, except in our case, the dual bounds generated by the functions $\nu_t$ play a specific role in the expansion step. Let $\mathcal X = \cup_{t}  \, \mathcal X_t$ be the set of all possible state nodes and let $\mathcal Y = \cup_t \, \mathcal Y_t$ be the set of all possible state-action nodes. At any iteration $n \ge 0$, our tree $\mathscr T^n = (n, \mathcal X^n, \mathcal Y^n, \bar{V}^n, \bar{Q}^n, \bar{u}^n, v^n,l^n)$ is described by the set $\mathcal X^n \subseteq \mathcal X$ of expanded state nodes, the set $\mathcal Y^n \subseteq \mathcal Y$ of expanded state-action nodes, the value function approximations $\bar{V}^n : \mathcal X \rightarrow \mathbb R$ and $\bar{Q}^n : \mathcal Y \rightarrow \mathbb R$, the estimated upper bounds $\bar{u}^n : \mathcal Y \rightarrow \mathbb R$, the number of visits $v^n : \mathcal X \, \cup \, \mathcal Y \rightarrow \mathbb R$ to expanded nodes, and the number of information relaxation upper bounds, or ``lookaheads,'' $l^n : \mathcal Y \rightarrow \mathbb R$ performed on unexpanded nodes. The terminology ``lookahead'' is used to mean a stochastic evaluation of the dual upper bound given in Proposition \ref{prop:upperbound}. In other words, we ``lookahead'' into the future and then exploit this information (thereby relaxing nonanticipativity) to produce an upper bound.

The root node of $\mathscr T^n$, for all $n$, is $\x_0 = s_0$. Recall that any node contains full information regarding the path from the initial state $\mathbf{x}_0 = s_0$. Therefore, in this paper, the edges of the tree are implied and we do not need to explicitly refer to them; however, we will use the following notation. For a state node $\x \in \mathcal X^n$, let $\mathcal Y^n(\x)$ be the child state-action nodes (i.e., already expanded nodes) of $\x$ at iteration $n$ (dependence on $\mathscr T^n$ is suppressed) and $\tilde{\mathcal Y}^n(\x)$ be the unexpanded state-action nodes of $\x$:
\[
\mathcal Y^n(\x) = \{ (\x,a') :  a' \in \mathcal A_\x,\, (\x,a') \in \mathcal Y^n  \},\quad \tilde{\mathcal Y}^n(\x) = \{ (\x,a') :  a' \in \mathcal A_\x,\, (\x,a') \not \in \mathcal Y^n(\x)  \}.
\]
Furthermore, we write $\tilde{\mathcal Y}^n = \cup_{\x \in \mathcal X^n} \tilde{\mathcal Y}^n(\x)$.

Similarly, for $\mathbf{y} = (s_0,a_0,\ldots,s_t,a_t) \in \mathcal Y^n$, let $\mathcal X^n(\mathbf{y})$ be the child state nodes of $\mathbf{y}$ and $\tilde{\mathcal X}^n(\y)$ be the unexpanded state nodes of $\y$:
\[
\mathcal X^n(\y) = \{ (\y,s) :  s \in \mathcal S,\, (\y,s) \in \mathcal X^n  \}, \quad \tilde{\mathcal X}^n(\y) = \{ (\y,s) :  s \in \mathcal S,\, (\y,s) \not \in \mathcal X^n(\y)\}.
\]

For mathematical convenience, we have $\bar{V}^0$, $\bar{Q}^0$, $\bar{u}^0$, $v^0$, and $l^0$ taking the value zero for all elements of their respective domains. For each $\x \in \mathcal X^n$ and $\y \in \mathcal Y^n$, let $\bar{V}^n(\x)$ and $\bar{Q}^n(\y)$ represent the estimates of $V^*(\x)$ and $Q^*(\y)$, respectively. Note that although $\bar{V}^n(\x)$ is defined (and equals zero) prior to the expansion of $\x$, it does not gain meaning until $\x \in \mathcal X^n$. The same holds for the other quantities.

Each unexpanded state node $\y^\text{u} \in \tilde{\mathcal Y}^n$ is associated with an estimated dual upper bound $\bar{u}^n(\y^\text{u})$. A state node $\x$ is called \emph{expandable} on iteration $n$ if $\tilde{\mathcal Y}^n(\x)$ is nonempty. Similarly, a state-action node $\y$ is \emph{expandable} on iteration $n$ if $\tilde{\mathcal X}^n(\y)$ is nonempty. In addition, let $v^n(\x)$ and $v^n(\y)$ count the number of times that $\x$ and $\y$ are visited by the selection policy (so $v^n$ becomes positive after expansion). The tally $l^n(\y)$ counts the number of dual lookaheads performed at each unexpanded state. We also need stepsizes $\alpha^n(\x)$ and $\alpha^n(\y)$ to track the estimates $\bar{V}^n(\x)$ generated by $\pid$ for leaf nodes $\x$ and $\bar{u}^n(\y)$ for leaf nodes $\y \in \tilde{\mathcal Y}^n$.


Lastly, we define two sets of \emph{progressive widening iterations}, $\mathcal N_{x} \subseteq \{0,1,2,\ldots\}$ and $\mathcal N_{y} \subseteq \{0,1,2,\ldots\}$.  When $v^n(\x) \in \mathcal N_x$, we consider expanding the state node $\x$ (i.e., adding a new state-action node stemming from $\x$), and when $v^n(\y) \in \mathcal N_y$, we consider expanding the state-action node $\y$ (i.e., adding a downstream state node stemming from $\y$).

\subsection{Selection} Let $\pis$ be a \emph{selection policy} that steers the algorithm down the \emph{current version} of the decision tree. It is independent from the rest of the system and depends only on the current state of the decision tree. We use the same notation for both types of nodes: for $\x \in \mathcal X^{n-1}$ and $\y \in \mathcal Y^{n-1}$, we have 
\[
\pi^{\textnormal{s}}(\x, \mathscr T^{n-1}) \in \mathcal Y^{n-1}(\x) \quad \text{and} \quad \pi^{\textnormal{s}}(\y, \mathscr T^{n-1}) \in \mathcal X^{n-1}(\y).
\]
Let us emphasize that $\pis$ contains no logic for expanding the tree and simply provides a path down the partial tree $\mathscr T^n$. The most popular MCTS implementations \citep{Chang2005,Kocsis2006} use the UCB1 policy \citep{Auer2002} for $\pis$ when acting on state nodes. The UCB1 policy balances exploration and exploitation by selecting the state-action node $\y$ by solving
\begin{equation}
\pis(\x,\mathscr T^{n-1})  \in \argmax_{\y \in \mathcal Y^{n-1}(\x)} \, \bar{Q}^{n-1}(\y) + \sqrt{\frac{2 \ln \sum_{\y' \in \mathcal Y^{n-1}(\x)} v^{n-1}(\y') }{ v^{n-1}(\y)}}.
\label{eq:UCB1}
\end{equation}
The second term is an ``exploration bonus'' which decreases as nodes are visited.
Other multi-armed bandit policies may also be used; for example, we may instead prefer to implement an $\epsilon$-greedy policy where we exploit with probability $1-\epsilon$ and explore with probability (w.p.) $\epsilon$:
\[
\pis(\x,\mathscr T^{n-1}) = \begin{cases} \argmax_{\y \in \mathcal Y^{n-1}(\x)} \, \bar{Q}(\y) & \text{w.p.} \quad 1-\epsilon,\\
\text{a random element from $\mathcal Y^{n-1}(\x)$} & \text{w.p.} \quad \epsilon.
\end{cases}
\]

When acting on state-action nodes, $\pis$ selects a downstream state node; for example, given $\y_t = (s_0,a_0,\ldots,s_t,a_t)$, the selection policy $\pis(\y_t,\mathscr T^{n-1})$ may select $\x_{t+1} = (s_0,a_1,\ldots,s_{t+1}) \in \mathcal X^{n-1}(\y_t)$ with probability $\P(S_{t+1} = s_{t+1} \, | \, \y_t)$, normalized by the total probability of reaching expanded nodes $\mathcal X^n(\y_t)$. We require the condition that once all downstream states are expanded, the sampling probabilities match the transition probabilities of the original MDP. We now summarize the \emph{selection phase} of Primal-Dual MCTS.
\begin{itemize}
\item Start at the root node and descend the tree using the selection policy $\pis$ until one of the following is reached: Condition (S1), an expandable state node $\x$ with $v^n(\x) \in \mathcal N_x$; Condition (S2), an expandable state-action node $\y$ with $v^n(\y) \in \mathcal N_y$; or Condition (S3), a leaf state node $\x$ is reached.
\item If the selection policy ends with conditions (S1) or (S2), then we move on to the \emph{expansion step}. Otherwise, we move on to the \emph{simulation and backpropagation steps}.
\end{itemize}

\subsection{Expansion}
\textbf{Case 1:} First, suppose that on iteration $n$, the selection phase of the algorithm returns $\x_{\taue}^n = (s_0,a_0,\ldots,s_{\taue})$ to be expanded, for some $\taue \in \mathcal T$. Due to the possibly large set of unexpanded actions, we first sample a subset of candidate actions (e.g., a set of $k$ actions selected uniformly at random from those in $\mathcal A$ that have not been expanded). Application specific heuristics may be employed when sampling the set of candidates. Then, for each candidate, we perform a lookahead to obtain an estimate of the perfect information relaxation dual upper bound. The lookahead is evaluated by solving a deterministic optimization problem on one sample path of the random process $\{W_t\}$. In the most general case, this is a deterministic dynamic program. However, other formulations may be more natural and/or easier to solve for some applications. If the contribution function is linear, the deterministic problem could be as simple as a linear program (for example, the asset acquisition problem class described in \cite{Nascimento2009a}). See also \cite{Al-Kanj2016} for an example where the information relaxation is a mixed-integer linear program. The resulting stochastic upper bound is then smoothed with the previous estimate via the stepsize $\alpha^n(\x_\taue^n)$. We select the action with the highest upper bound to expand, but only if the upper bound is larger than the current best value function $\bar{Q}^n$. Otherwise, we skip the expansion step because our estimates tell us that none of the candidate actions are optimal. The following steps comprise of the \emph{expansion phase} of Primal-Dual MCTS for a state node $\x_{\taue}^n$.
\begin{itemize}
\item Sample a subset of candidate actions according to a pre-specified sampling policy $\pia(\x_\taue^n, \mathscr T^{n-1}) \subseteq \mathcal A_{\x_\taue^n}$ and consider those actions that are unexpanded:
\[
\Atilde^n(\x_\taue^n) = \pia(\x_\taue^n, \mathscr T^{n-1})  \cap \{ a \in \mathcal A_{\x_\taue^n}: (\x_\taue^n,a) \in \tilde{\mathcal Y}^n(\x_\taue^n)  \}. 
\]
\item Obtain a single sample path $\W_{\taue+1,T}^n = (W_{\taue+1}^n,\ldots W_T^n)$ of the exogenous information process. For each candidate action $a \in \tilde{\mathcal A}^n(\x_\taue^n)$, compute the optimal value of the deterministic optimization ``inner'' problem of (\ref{eq:upperbound}):
\begin{align*}
\hat{u}^n(&\x_\taue^n,a) = \\
&c_\taue(s,a,W^n_{\taue+1}) +  \max_{\a} \bigl[ h_{\taue+1}\bigl(S_{\taue+1},\a,\W^n_{\taue+1,T}\bigr) - z_{\taue+1}^{\nu}\bigl(S_{\taue+1},\a,\W^n_{\taue+1,T}\bigr) \bigr].
\end{align*}
\item For each candidate action $a \in \tilde{\mathcal A}^n(\x_\taue^n)$, smooth the newest observation of the upper bound with the previous estimate via a stochastic gradient step:
\begin{equation}
\bar{u}^n(\x_\taue^n,a) = (1-\alpha^n(\x_\taue^n,a))\, \bar{u}^{n-1}(\x_\taue^n,a) + \alpha^n(\x_\taue^n,a) \, \hat{u}^n(\x_\taue^n,a).
\label{eq:duallookahead}
\end{equation}
State-action nodes $\y$ elsewhere in the tree that are not considered for expansion retain the same upper bound estimates, i.e., $\bar{u}^n(\y) = \bar{u}^{n-1}(\y)$.
\item Let $a^{n} = \argmax_{a \in \tilde{\mathcal A}^n(\x_\taue^n)} \bar{u}^n(\x_\taue^n,a )$ be the candidate action with the best dual upper bound. If no candidate is better than the current best, i.e., $\bar{u}^n(\x_\taue^n, a^n) \le \bar{V}^{n-1}(\x_{\taue}^n)$, then we skip this potential expansion and return to the \emph{selection phase} to continue down the tree.
\item Otherwise, if the candidate is better than the current best, i.e., $\bar{u}^n(\x_\taue^n, a^n) > \bar{V}^{n-1}(\x_{\taue}^n)$, then we \emph{expand} action $a^n$ by adding the node $\y_\taue^n = (\x_\taue^n,a^n)$ as a child of $\x_\taue^n$. We then immediately sample a downstream state $\x_{\taue+1}^n$ using $\pis$ from the set $\tilde{\mathcal X}^n(\y_\taue^n)$ and add it as a child of $\y_\taue^n$ (every state-action expansion triggers a state expansion). After doing so, we are ready to move on to the \emph{simulation and backpropagation phase} from the leaf node $\x_{\taue+1}^n$.
\end{itemize}

\noindent \textbf{Case 2:} Now suppose that we entered the \emph{expansion phase} via a state-action node $\y_\taue^n$. In this case, we simply sample a single state $\x_{\taue+1}^n = (\y_\taue^n, s_{\taue+1})$ from $\tilde{\mathcal X}^n(\y_\taue^n)$ such that
\[
\P(S_{\taue+1} = s_{\taue+1} \, | \, \y_\taue^n) > 0
\] and add it as a child of $\y_\taue^n$. Next, we continue to the \emph{simulation and backpropagation phase} from the leaf node $\x_{\taue+1}^n$.

\begin{figure}[h]
   \includegraphics[width=1.05\textwidth]{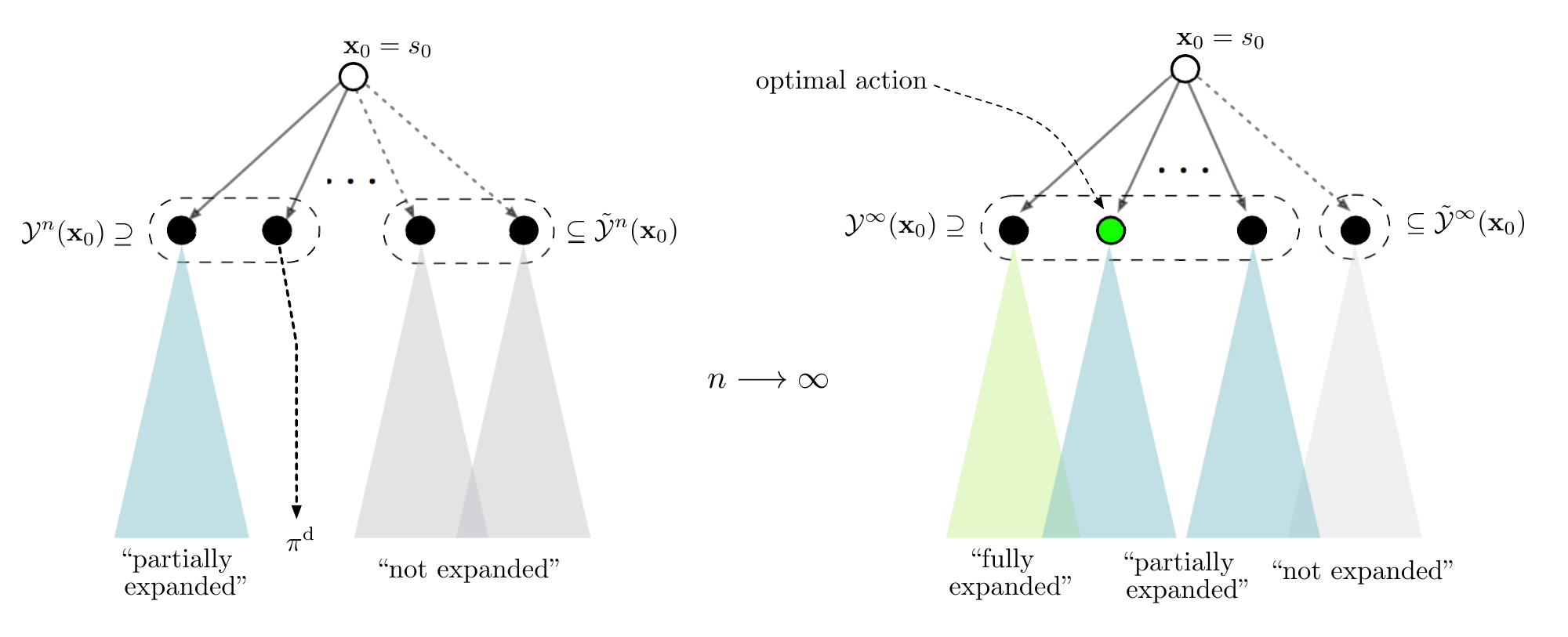}
    \vspace{-3pt}
    \caption{Properties of the Primal-Dual MCTS Algorithm}
    \label{fig:trees}
\end{figure}

\subsection{Simulation and Backpropagation}
\label{sec:simulationback}
We are now at a leaf node $\x_{\taus}^n = (s_0,a_0,\ldots,s_\taus)$, for some $\taus \le T$. At this point, we cannot descend further into the tree so we proceed to the \emph{simulation and backpropagation phase}. The last two steps of the algorithm are relatively simple: first, we run the default policy to produce an estimate of the leaf node's value and then update the values ``up'' the tree via equations resembling (\ref{eq:bellman1}) and (\ref{eq:bellman2}). The steps are as follows.
\begin{itemize}
\item  Obtain a single sample path $\W_{\taus+1,T}^n = (W_{\taus+1}^n,\ldots W_T^n)$ of the exogenous information process and using the default policy $\pid$, compute the value estimate
\begin{equation}
\hat{V}^n(\x_\taus^n) = h_\taus(s_\taus,\pid,\W_{\taus+1,T}^n) \, \mathbf{1}_{\{\taus<T\}}.
\label{eq:simobs}
\end{equation}
If $\taus = T$, then the value estimate is simply the terminal value of zero. The value of the leaf node is updated by taking a stochastic gradient step that smooths the new observation with previous observations according to the equation
\begin{equation*}
\bar{V}^n(\x_\taus^n) = \left( 1 - \alpha^n(\x_\taus^n) \right) \, \bar{V}^{n-1}(\x_\taus^n)  + \alpha^n(\x_\taus^n) \, \hat{V}^n(\x_\taus^n).
\label{eq:simulation}
\end{equation*}
\item After simulation, we backpropagate the information up the tree. Working backwards from the leaf node, we can extract a ``path,'' or a sequence of state and state-action nodes $\x_\taus^n, \y_{\taus-1}^n, \ldots, \x_1^n, \y_0^n, \x_0^n$ (each of these elements is a ``subsequence'' of the vector $\x_\taus^n = (s_0, a_0^n, \ldots, s_\taus^n)$, starting with $s_0$). For $t = \taus-1, \taus-2, \ldots, 0$, the backpropagation equations are:
\begin{align}
\bar{Q}^n(\y_t^n) &= \bar{Q}^{n-1}(\y_t^n) + \frac{1}{v^n(\y_{t}^n) } \bigl[\bar{V}^n(\x^n_{t+1}) - \bar{Q}^n(\y_t^n)\bigr],\label{eq:backprop1}\\
\tilde{V}^{n}(\x_t^n) &= \tilde{V}^{n-1}(\x_t^n) + \frac{1}{v^n(\x_{t}^n) } \bigl[\bar{Q}^n(\y_t^n) - \tilde{V}^{n}(\x_t^n)\bigr],\label{eq:backprop2}\\
\bar{V}^n(\x_t^n) &=  (1-\lambda^n) \, \tilde{V}^{n}(\x_t^n) + \lambda^n \max_{\y_t} \bar{Q}^n(\y_{t}),\label{eq:backprop3}
\end{align}
where $\y_t \in \mathcal Y^n(\x_t^n)$ and $\lambda^n \in [0,1]$ is a mixture parameter. Nodes $\x$ and $\y$ that are not part of the path down the tree retain their values, i.e.,
\begin{equation}
\bar{V}^n(\x) = \bar{V}^{n-1}(\x) \quad \text{and} \quad \bar{Q}^n(\y) = \bar{Q}^{n-1}(\y).
\label{eq:backprop4}
\end{equation}
\end{itemize}
The first update (\ref{eq:backprop1}) maintains the estimates of the state-action value function as weighted averages of child node values. The second update (\ref{eq:backprop2}) similarly performs a recursive averaging scheme for the state nodes and finally, the third update (\ref{eq:backprop3}) sets the value of a state node to be a mixture between the weighted average of its child state-action node values and the maximum value of its child state-action nodes.

The naive update for $\bar{V}^n$ is to simply take the maximum over the state-action nodes (i.e., following the Bellman equation), removing the need to track $\tilde{V}^n$. Empirical evidence from \cite{Coulom2007}, however, shows that this type of update can create instability; furthermore, the authors state that ``the mean operator is more accurate when the number of simulations is low, and the max operator is more accurate when the number of simulations is high.'' Taking this recommendation, we impose the property that $\lambda^n \nearrow 1$ so that asymptotically we achieve the Bellman update yet allow for the averaging scheme to create stability in the earlier iterations. The update (\ref{eq:backprop3}) is similar to ``mix'' backup suggested by \cite{Coulom2007} which achieves superior empirical performance. 

The end of the simulation and backpropagation phase marks the conclusion of one iteration of the Primal-Dual MCTS algorithm. We now return to the root node and begin a new selection phase. Algorithm \ref{alg:pd} gives a concise summary of Primal-Dual MCTS. Moreover, Figure \ref{fig:trees} illustrates some aspects of the algorithm and emphasizes two key properties:
\begin{itemize}
\item The utilization of dual bounds allows entire subtrees to be ignored (even in the limit), thereby providing potentially significant computational savings.
\item The optimal action at the root node can be found without its subtree necessarily being fully expanded.
\end{itemize}
We will analyze these properties in the next section, but we first present an example that illustrates in detail the steps taken during the expansion phase.

\IncMargin{1em}
\begin{algorithm}

  \SetKwInput{Input}{Input}
  \SetKwInput{Output}{Output}
  \DontPrintSemicolon
\Indm  
  \Input{An initial state node $\x_0$, a default policy $\pid$, a selection policy $\pis$, a candidate sampling policy $\pia$, a stepsize rule $\{\alpha^n\}$, a backpropagation mixture scheme $\{\lambda^n\}$.}
  \BlankLine
  \Output{Partial decision trees $\{\mathscr T^n\}$.}
\Indp
  \BlankLine
  \For{$n = 1, 2, \ldots$}{
  \BlankLine
  \nl  run \texttt{Selection} phase with policy $\pis$ from $\x_0$ and return either condition (S1) with $\x_{\taue}^n$, (S2) with $\y_{\taue}^n$, or (S3) with $\x_{\taus}^n$.\;
  \BlankLine
    \uIf{condition \textnormal{(S1)}}{
      \BlankLine
  \nl			run \texttt{Case 1} of \texttt{Expansion} phase with policy $\pia$ at state node $\x_{\taue}^n$ and return leaf node $\x_{\taus}^n = \x_{\taue+1}^n.$
    \BlankLine

  		}
  		\ElseIf{condition \textnormal{(S2)}}{
  		  \BlankLine

  \nl			run \texttt{Case 2} of \texttt{Expansion} phase at state-action node $\y_{\taue}^n$ and return leaf node $\x_{\taus}^n = \x_{\taue+1}^n$.
    \BlankLine

  		}

  \BlankLine
  \nl run \texttt{Simulation and Backpropagation} phase from leaf node $\x_{\taus}^n$.
  \BlankLine

  }
    \caption{Primal-Dual Monte Carlo Tree Search}
    \label{alg:pd}
\end{algorithm}
\DecMargin{1em}

\begin{example}[Shortest Path with Random Edge Costs]
In this example, we consider applying the Primal-Dual MCTS to a shortest path problem with random edge costs (note that the algorithm is stated for maximization while shortest path is a minimization problem). The graph used for this example is shown in Figure \ref{fig:subshortest}. An agent starts at vertex 1 and aims to reach vertex 6 at minimum expected cumulative cost. The cost for edge $e_{ij}$ (from vertex $i$ to $j$) is distributed $\mathcal N(\mu_{ij},\sigma^2_{ij})$ and independent from the costs of other edges and independent across time. At every decision epoch, the agent chooses an edge to traverse out of the current vertex without knowing the actual costs. After the decision is made, a realization of edge costs is revealed and the agent incurs the one-stage cost associated with the traversed edge.

\begin{figure}[h]
        \centering
        \begin{subfigure}[b]{0.49\textwidth}
                \centering
                \includegraphics[width=\textwidth]{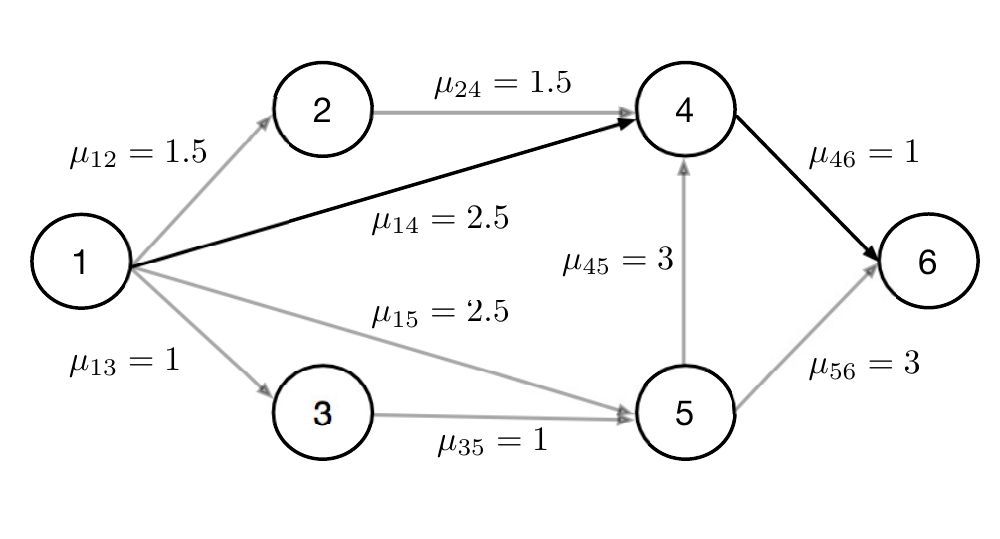}
                \caption{Graph with Mean Costs}
                \label{fig:subshortest}
        \end{subfigure}
        \begin{subfigure}[b]{0.49\textwidth}
                \centering
                \includegraphics[width=\textwidth]{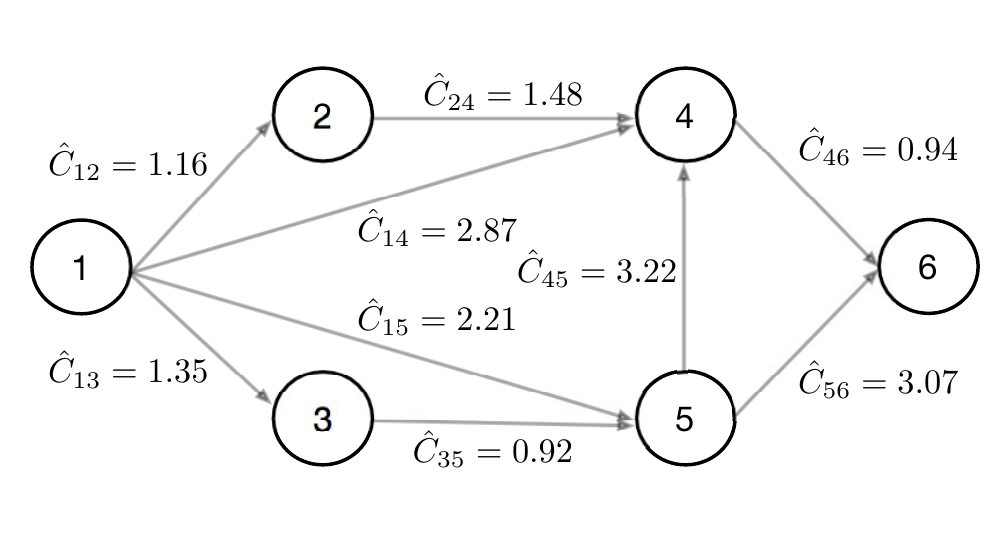}
                \caption{Graph with Sampled Costs}
                \label{fig:subshortest2}
        \end{subfigure}
        \caption{Shortest Path Problem with Random Edge Costs}
        \label{fig:shortest}
\end{figure}

The mean of the cost distributions are also shown in Figure \ref{fig:subshortest} and we assume that $\sigma_{ij} = 0.25$. The optimal path is $1 \rightarrow 4 \rightarrow 6$, which achieves an expected cumulative cost of 3.5. Consider applying Primal-Dual MCTS at vertex 1, meaning that we are choosing between traversing edges $e_{12}$, $e_{13}$, $e_{14}$, and $e_{15}$. The shortest paths after choosing $e_{12}$, $e_{13}$, and $e_{15}$ are $1 \rightarrow 2 \rightarrow 4 \rightarrow 6$ (cost of 4), $1 \rightarrow 3 \rightarrow 5 \rightarrow 6$ (cost of 5), and $1 \rightarrow 5 \rightarrow 6$ (cost of 5.5), respectively. Hence, $Q^*(1,e_{12}) = 4$, $Q^*(1,e_{13}) = 5$, $Q^*(1,e_{14})=3.5$, and $Q^*(1,e_{15}) = 5.5$.

We now illustrate several consecutive \emph{expansion steps} (this means that there are non-expansion steps in-between that are not shown) from the point of view of vertex 1, where there are four possible actions, $e_{1i}$ for $i = 2, 3, 4, 5$. On every expansion step, we use one sample of exogenous information (costs) to perform the information relaxation step and compute a standard dual (lower) bound. For simplicity, suppose that on every expansion step, we see the same sample of costs that are shown in Figure \ref{fig:subshortest2}. By finding the shortest paths in the graph with sampled costs, the sampled dual bounds are thus given by $\bar{u}^n(1,e_{12}) = 3.58$, $\bar{u}^n(1,e_{13}) = 5.34$, $\bar{u}^n(1,e_{14})=3.81$, and $\bar{u}^n(1,e_{15}) = 5.28$ (assuming the initial stepsize is 1). Figure \ref{fig:expansion} illustrates the expansion process.
\begin{enumerate}
\item In the first expansion, nothing has been expanded so we simply expand edge $e_{12}$ because it has the lowest dual bound. Note that this is not the optimal action; the optimistic dual bound is the result of noise.
\item After some iterations, learning has occurred for $\bar{Q}^n(1,e_{12})$ and it is currently estimated to be 3.97. We expand $e_{14}$ because it is the only unexpanded action with a dual bound that is better than 3.97. This is the optimal action.
\item In the last step of Figure \ref{fig:expansion}, no actions are expanded because their dual bounds indicate that they are no better than the currently expanded actions.
\end{enumerate}

\begin{figure}[h]
   \hspace{-15pt}\includegraphics[width=1.1\textwidth]{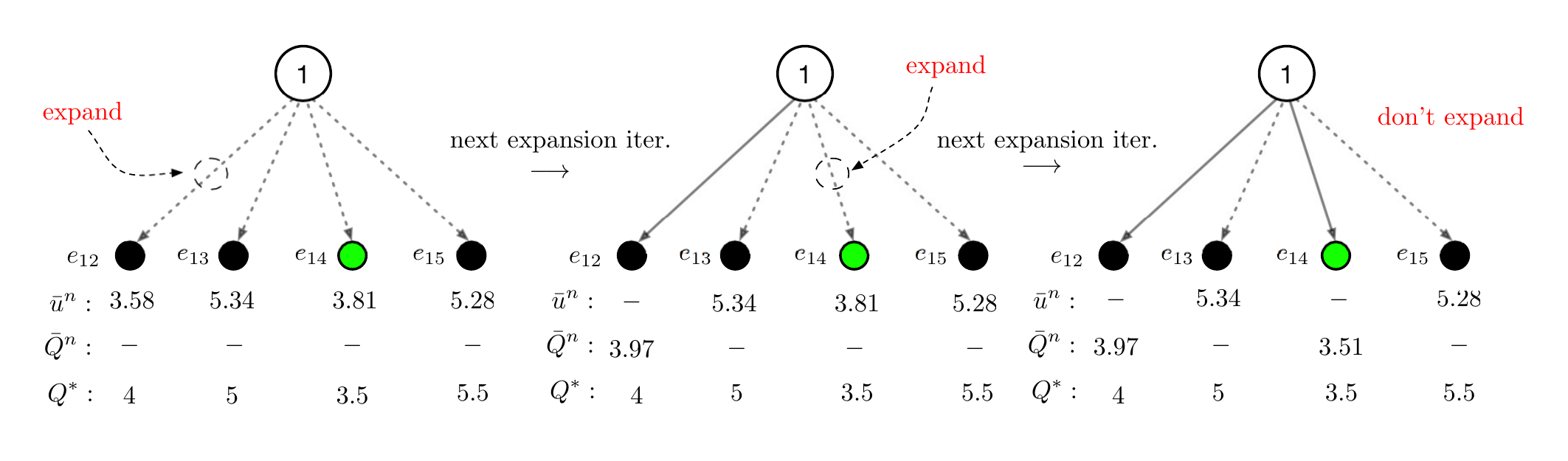}
    \vspace{-10pt}
    \caption{Expansion Steps for the Example Problem}
    \label{fig:expansion}
\end{figure}

\end{example}

\section{Analysis of Convergence}
\label{sec:conv}
Let $\mathscr T^\infty$ be the ``limiting partial decision tree'' as iterations $n \rightarrow \infty$. Similarly, we use the notation $\mathcal X^\infty$, $\mathcal X^\infty(\y)$, $\tilde{\mathcal X}^\infty(\y)$, $\mathcal Y^\infty$, $\mathcal Y^\infty(\x)$, and $\tilde{\mathcal Y}^\infty(\x)$ to describe the random sets of expanded and unexpanded nodes of the tree in the limit, analogous to the notation for a finite iteration $n$. Given that there are a finite number of nodes and that the cardinality of these sets is monotonic with respect to $n$, it is clear that these limiting sets are well-defined. 

Recall that each iteration of the algorithm generates a leaf node $\x_{\tau_s}^n$, which also represents the path down the tree for iteration $n$. Before we begin the convergence analysis, let us state a few assumptions.
\begin{assumption}
Assume the following hold. 
\begin{enumerate}[label=(\roman*),labelindent=1in]

\item There exists an $\epsilon^\textnormal{s} > 0$ such that given any tree $\mathscr T$ containing a state node $\x_t \in \mathcal X$ and a state-action node $\y_t=(\x_t,a_t)$ with $a_t \in \mathcal A_{\x_t}$, it holds that $\P( \pis(\x_t,\mathscr T) = \y_t)  \ge \epsilon^\textnormal{s}$.
\item Given a tree $\mathscr T$ containing a state action node $\y_t$, if all child state nodes of $\y_t$ have been expanded, then
\[
\P(\pis(\y_t,\mathscr T) = \x_{t+1}) = \P(S_{t+1} = s_{t+1} \, | \, \y_t)
\]
where $\x_{t+1} = (\y_t,s_{t+1})$. This means that sampling eventually occurs according to the true distribution of $S_{t+1}$.

\item There exists an $\epsilon^\textnormal{a} > 0$ such that given any tree $\mathscr T$ containing a state node $\x_t \in \mathcal X$ and action $a_t \in \mathcal A_{\x_t}$, it holds that $\P( a_t \in \pia(\x_t,\mathscr T))  \ge \epsilon^\textnormal{a}$.
\item There are an infinite number of progressive widening iterations: $|\mathcal N_x| = |\mathcal N_y| = \infty$. 
\item For any state node $\mathbf{x_t} \in \mathcal X$ and action $a_t$, the stepsize $\alpha^n(\x_t,a_t)$ takes the form
\[
\alpha^n(\x_t,a_t) = \tilde{\alpha}^n \, \indicate{\x_t \in  \x_{\tau_s}^n} \,  \indicate{v^n(\x_t) \in \, \mathcal N_x} \, \indicate{a_t \in \pia(\x_t,\mathscr T^n)}, 
\]
for some possibly random sequence $\tilde{\alpha}^n$. This means that whenever the dual lookahead update (\ref{eq:duallookahead}) is not performed, the stepsize is zero. In addition, the stepsize sequence satisfies
\[
\sum_{n=0}^\infty \alpha^n(\x_t,a) = \infty \; \; a.s. \quad \text{and} \quad \sum_{n=0}^\infty \alpha^n(\x_t,a)^2 < \infty \; \; a.s.,
\]
the standard stochastic approximation assumptions.
\item As $n \rightarrow \infty$, the backpropagation mixture parameter $\lambda^n \rightarrow 1$.
\end{enumerate}
\label{ass:one}
\end{assumption}

An example of a stepsize sequence that satisfies Assumption \ref{ass:one}(v) is $1/l^n(\x_t,a_t)$. We now use various aspects of Assumption \ref{ass:one} to demonstrate that expanded nodes within the decision tree are visited infinitely often. This is, of course, crucial in proving convergence, but due to the use of dual bounds, we only require that the limiting \emph{partial decision tree} be visited infinitely often. Previous results in the literature require this property on the fully expanded tree.

\begin{restatable}{lemma}{lemio}
Let $\x \in \mathcal X$ be a state node such that $\P(\x \in \mathcal X^\infty) > 0$. Under Assumption \ref{ass:one}, it holds that $v^n(\x) \rightarrow \infty$ almost everywhere on $\{\x \in \mathcal X^\infty\}$. Let $\y \in \mathcal Y$ be a state-action node such that $\P(\y \in \mathcal Y^\infty) > 0$. Similarly, we have $v^n(\y) \rightarrow \infty$ almost everywhere on $\{\y \in \mathcal Y^\infty\}$. Finally, let $\y' \in \mathcal Y$ be such that $\P(\y' \in \tilde{\mathcal Y}^\infty) > 0$. Then, $l^n(\y') \rightarrow \infty$ almost everywhere on $\{ \y' \in \tilde{\mathcal Y}^\infty$\}, i.e., the dual lookahead for the unexpanded state-action node $\y' \in \tilde{\mathcal Y}^\infty(\x')$ is performed infinitely often. 
\label{lem:io}
\end{restatable}
\begin{proof}
See Appendix \ref{sec:proofs}.
\end{proof}

The next lemma reveals the central property of Primal-Dual MCTS (under the assumption that all relevant values converge appropriately): for any expanded state node, its corresponding optimal state-action node is expanded. In other words, if a particular action is never expanded, then it must be suboptimal. 

\begin{restatable}{lemma}{lemargmax}
Consider a state node $\x_t \in \mathcal X$. Consider the event on which $\x_t \in \mathcal X^\infty$ and the following hold:
\begin{enumerate}[label=(\roman*),labelindent=1in]

\item $\bar{Q}^n(\y_t) \rightarrow Q^*(\y_t)$ for each expanded $\y_t \in \mathcal Y^\infty(\x_t)$,
\item $\bar{u}^n(\y_t') \rightarrow u^{\nu}(\y_t')$ for each unexpanded $\y_t' \in \tilde{\mathcal Y}^\infty(\x_t)$.
\end{enumerate}
Then, on this event, there is a state-action node $\y_t^* = (\x_t,a^*_t) \in \mathcal Y^\infty(\x_t)$ associated with an optimal action $a_t^* \in \argmax_{a \in \mathcal A} Q_t^*(s_t,a)$.
\label{lem:argmax}
\end{restatable}
\begin{proof}[Sketch of Proof:]
The essential idea of the proof is as follows. If all optimal actions are unexpanded and the assumptions of the lemma hold, then eventually, the dual bound associated with an unexpanded optimal action must upper bound the values associated with the expanded actions (all of which are suboptimal). Thus, given the design of our expansion strategy to explore actions with high dual upper bounds, it follows that an optimal action must eventually be expanded.
Appendix \ref{sec:proofs} gives the technical details of the proof.
\end{proof}

We are now ready to state the main theorem, which shows consistency of the proposed procedure. We remark that it is never required that $\mathcal X^\infty_t = \mathcal X_t$ or $\mathcal Y^\infty_t = \mathcal Y_t$. In other words, an important feature of Primal-Dual MCTS is that the tree does not need to be fully expanded in the limit, as we alluded to earlier in Figure \ref{fig:trees}.

\begin{restatable}{theorem}{mainthm}
Under Assumption \ref{ass:one}, the Primal-Dual MCTS procedure converges at the root node (initial state) in two ways:
\[
\bar{V}^n(\x_0) \rightarrow V^*(\x_0) \; \; a.s. \quad \textnormal{and} \quad \limsup_{n \rightarrow \infty} \, \argmax_{\y \in \mathcal Y^n(\x_0)} \bar{Q}^n(\y) \subseteq \argmax_{\y_0 = (\x_0,a)} Q^*(\y_0) \; \; a.s.
\]
meaning that the value of the node $\x_0$ converges to the optimal value and that an optimal action is both expanded and identified.
\label{thm:main}
\end{restatable}
\begin{proof}[Sketch of Proof:]
The proof of the main theorem places the results established in the previous lemmas in an induction framework that moves up the tree, starting from state nodes $\x_T \in \mathcal X_T$. We first establish the following convergence results:
\begin{align*}
\bar{Q}^n(\y_{t}) &\rightarrow Q^*(\y_{t}) \, \indicate{\y_{t} \in \mathcal Y_{t}^\infty} \quad a.s.,\\
\bar{u}^n(\y_{t}) &\rightarrow u^{\nu}(\y_t) \, \indicate{\y_{t} \in \tilde{\mathcal Y}^\infty_t} \quad a.s.,
\end{align*}
after which Lemma \ref{lem:argmax} can be invoked to conclude $\bar{V}^n(\x_{t}) \rightarrow V^*(\x_{t}) \, \indicate{\x_{t} \in \mathcal X_{t}^\infty}$ almost surely. The full details are given in Appendix \ref{sec:proofs}.
\end{proof}

 In addition, let us comment that Assumption \ref{ass:one}(v) could also be replaced with an alternative condition on the selection policy; for example, if the visits concentrate on the optimal action asymptotically, then the average over the state-action values would converge to the optimal value. \cite{Chang2005} takes this approach by utilizing results from the multi-armed bandit literature \citep{Bubeck2012}. The \cite{Chang2005} method, although similar to MCTS (and indeed served as inspiration for MCTS), differs from our description of MCTS in a crucial way: the levels or stages of the tree are never updated together. The algorithm is a one-stage method which calls itself in a recursive fashion starting from $t=T$. When nodes at a particular stage $t$ are updated, the value function approximation for stage $t+1$ has already been fixed; hence, results from the multi-armed bandit literature can be directly applied. Unfortunately, this is not the case for MCTS, where updates to the entire tree are made at every iteration. 

\section{Driver Behavior on a Ride-Sharing Platform}
\label{sec:num}
In this section, we show numerical results of applying Primal-Dual MCTS on a model of driver behavior on a ride-sharing platform (e.g., Uber and Lyft). Our motivation for studying this problem is due to the importance of incorporating the aspect of driver decisions into fleet simulation models. Such large-scale models of the entire system operations can aid in making \emph{platform-level} decisions, including (1) spatial dynamic pricing for riders (Uber's ``surge pricing''), (2) dynamic wages/incentives for drivers (Uber's ``earnings boost''), and (3) the integration of autonomous vehicles with traditional human drivers (e.g., in Pittsburgh, Pennsylvania). Since optimal decisions from the driver point of view intimately depend on parameters (e.g., prices) determined by the platform, we envision that the problem studied here is a crucial building block within a higher level simulation model. Experimental testing suggests that the new version of MCTS produces deeper trees and reduced sensitivity to the size of the action space.
\subsection{Markov Decision Process Model}
The operating region is represented as a connected graph consisting of set of locations $\mathcal L = \{1, 2, \ldots, M\}$, and if $i, j \in \mathcal L$ are ``adjacent,'' then there exists an edge $e_{ij}$ connecting them. Let $\mathcal L(i)$ be the set of adjacent locations to $i$, including $i$ itself. The units of time are in increments of a ``minimum trip length,'' where a trip between adjacent nodes requires one time increment. A \emph{trip request} is an ordered pair $r = (i, j)$ where $i$ is the starting location and $j$ is the destination location. The status $\sigma_t$ of the driver can take several forms: ``idle,'' ``en-route to trip $r=(i,j)$,'' and ``with passenger, toward destination $l_j$.'' Respectively, these statuses are encoded as $\sigma_t = (0,\varnothing)$, $\sigma_t = (1,(i,j))$, and $\sigma_t = (2,j)$. The second element of $\sigma_t$ is interpreted as the current ``goal'' of the driver.

The sequence of events for an idle driver is as follows: (1) at the beginning of period $t$, the driver observes a set of requested trips $\mathcal R_t$; (2) the decision is made to either accept one of the trips in $\mathcal R_t$ \emph{or} reject all of them and move to an adjacent location; and (3) the driver's status is updated. If the driver is not idle (i.e. $\sigma_t \ne 0$), then there is no decision to be made and the current course is maintained.

Let $L_t \in \mathcal L$ be the location of the driver at time $t$. We assume that the stochastic process describing the sets of trip requests $\{\mathcal R_t\}$ is independent across time and $|\mathcal R_t| \le \mathcal R_\textnormal{max}$. Thus, the state variable of the MDP is $S_t = (L_t, \sigma_t, \mathcal R_t)$ and let $\mathcal S$ be the state space. The set of available decisions is given by
\[
\mathcal A_{S_t} = \begin{cases} \mathcal L(L_t) \, \cup \, \mathcal R_t &\textnormal{if} \quad \sigma_t = 0,\\
								 \varnothing							 &\textnormal{otherwise.}  \end{cases}
\]
Suppose that there is a well-defined choice of a shortest path between any two locations $i$ and $j$ given by a sequence of locations along the path $p(i,j) = (l_1, l_2, \ldots, l_{d(i,j)}) \in \mathcal L^{d(i,j)}$, where $d(i,j)$ is the distance of the shortest path between $i$ and $j$ and $l_{d(i,j)} = j$. We will use the notation $p_1(i,j)$ to represent the first element $l_1$ of the sequence $p(i,j)$, i.e., the next location to visit in order to reach $j$ starting from $i$.
Let $a_t \in \mathcal A_{S_t}$ be the decision made at time $t$. The transition of the driver's location is
\[
L_{t+1} = f_L(L_t,a_t)= \begin{cases} a_t &\textnormal{if} \quad \sigma_t = 0,\, a_t \in \mathcal L(L_t),\\
						p_1(L_t, i)	&\textnormal{if} \quad \sigma_t = 0,\, a_t = (i,j) \in \mathcal R_t \text{ or } \sigma_t = (1, (i,j)),\\
						p_1(L_t, j)	&\textnormal{if} \quad \sigma_t = (2,j).  \end{cases}
\]
Similarly, we can write the transition for the driver's status as
\[
\begin{aligned}
\sigma_{t+1} = &f_\sigma(\sigma_t,a_t)\\
&=\begin{cases} (0,\varnothing) &\textnormal{if} \quad \sigma_t = 0,\, a_t \in \mathcal L(L_t),\\
						(1,(i,j))	&\textnormal{if} \quad (\sigma_t = 0,\, a_t = (i,j) \in \mathcal R_t \text{ or } \sigma_t = (1, i,j)) \text{ and } d(L_t,j) > 1,\\
						(2,j)	&\textnormal{if} \quad (\sigma_t = 0,\, a_t = (i,j) \in \mathcal R_t \text{ or } \sigma_t = (1, i,j)) \text{ and } d(L_t,j) = 1,\\
						(2,j)	&\textnormal{if} \quad \sigma_t = (2,j) \text{ and } d(L_t,j) > 1,\\
						(0,\varnothing)	&\textnormal{if} \quad \sigma_t = (2,j) \text{ and } d(L_t,j) = 1.  \end{cases}
						\end{aligned}
\]
Suppose that the base fare is $w_\textnormal{base}$ and that the customer pays $w_\textnormal{dist}$ per unit distance traveled. The driver is profit-maximizing and contribution function is revenue generated from the customer with ``per mile'' travel costs $c$ subtracted (e.g., gas, vehicle depreciation) whenever the driver moves:
\[
c_t(S_t,a_t) = \bigl[w_\textnormal{base} + w_\textnormal{dist} \, d(i,j) \bigr]  \cdot \indicate{\sigma_t = 0,\, a_t \, = \, (i,j) \, \in \, \mathcal R_t}- c \cdot \indicate{L_t \ne f_L(L_t,a_t)}.
\]
The objective function is the previously stated (\ref{eq:mdpobj}), with $c_t(S_t,\pi(S_t))$ replacing the cost function $c_t(S_t,\pi(S_t),W_{t+1})$. Let $s=(l,\sigma,\mathcal R)$ and the corresponding Bellman optimality condition is 
\begin{align*}
&V_t^*(s) = \max_{a \in \mathcal A_s} \, c_t(s,a)   + \E \left[ V_{t+1}^*(f_L(l,a), f_\sigma(\sigma,a), \mathcal R_{t+1}) \right] \; \text{for all } s \in \mathcal S,\; t \in \mathcal T,\\
&V_T^*(s) = 0 \; \text{for all } s \in \mathcal S.
\end{align*}
Intuitively, the aim of the driver is to position the vehicle in the city and accepting  trip requests so that revenues can be collected without significant travel costs. Due to the size of the problem, standard MDP techniques for computing the optimal policy are intractable.
\subsection{Data, Problem Setup, \& Algorithm Details}
Our problem setting is the state of New Jersey and the experiments are performed on a dataset of all trips taken in one day throughout the state on a particular taxi service. We consider the situation where each unit of time corresponds to 15 minutes, i.e., the time it takes to travel the distance between any two adjacent locations. The driver is assumed to work for 10 hours a day, giving us $T = 40$. Over this time horizon, the dataset contains a total of 7{,}056 trips. The graph is built using realistic geography; each location in our model represents a $0.5 \times 0.5$ square mile area in the state of New Jersey. At each time step, the requests shown to the driver in the next time period, $\mathcal R_{t+1}$, are sampled uniformly from the data set, subject to $|\mathcal R_{t+1}| \le R_\textnormal{max}$.  The fare parameters are $w_\textnormal{base} = \$2.40$ (per trip), $w_\textnormal{dist} = \$2.02$ (per mile), and $c = \$0.05$ (per mile).

By making small adjustments to the graph and by changing $R_\textnormal{max}$, we consider three instances of the problem on the same dataset. Instance \texttt{D5} provides the driver with five possible decisions at every state: three neighboring locations and $R_\textnormal{max}= 2$ requests. Similarly, instance \texttt{D10} gives six neighboring locations and $R_\textnormal{max}= 4$ requests while instance \texttt{D15} gives 10 neighboring locations and $R_\textnormal{max}= 5$ requests.

We compare Primal-Dual MCTS and ``vanilla'' MCTS, which is the standard version of MCTS that does not sample the information relaxation bounds; all other parameters remain the same between the two versions.
Following existing work \citep{Chang2005,Kocsis2006}, the selection policy $\pis$ is chosen to be UCB1, as given in (\ref{eq:UCB1}). For this particular problem class, we found that the standard dual upper bound obtained by setting $\nu_t = 0$ was sufficient in improving the behavior of vanilla MCTS and we report results for this case. This information relaxation bound is computed by solving a \emph{deterministic dynamic program} after sampling a trajectory of ride requests from $t+1$ to $T$. The objective value of the dynamic program, is on average, an upper bound on the value of a state-action node. Our default policy $\pid$ involves a similar sequence of steps: (1) sample a future trajectory of ride requests, (2) implement the first decision, and (3) repeat from the next state. Although the computations are related to computing information relaxation bounds, this is a \emph{policy} (not a bound) that falls under the class of rolling-horizon procedures (see, e.g., \cite{Chand2002}).

\subsection{Numerical Results}
In this section, we discuss the numerical results obtained from applying Primal-Dual MCTS on the three instances of ride-sharing model, \texttt{D5}, \texttt{D10}, and \texttt{D15}, described above. It is important to first understand the qualitative implications of running Primal-Dual MCTS versus vanilla MCTS. The intuition is that because dual upper bounds are used in expansion decisions, fewer ``unnecessary'' expansions are performed. Therefore, it should be the case that in the same number of iterations with the same algorithm parameters, Primal-Dual MCTS generates deeper decision trees while vanilla MCTS generates wider decision trees. In MCTS applications, depth is preferred to breadth because it indicates a longer lookahead horizon that focuses on important parts of the tree.

\begin{figure}[!ht]
\small
        \centering
        \begin{subfigure}[b]{0.47\textwidth}
                \centering
                \includegraphics[width=\textwidth]{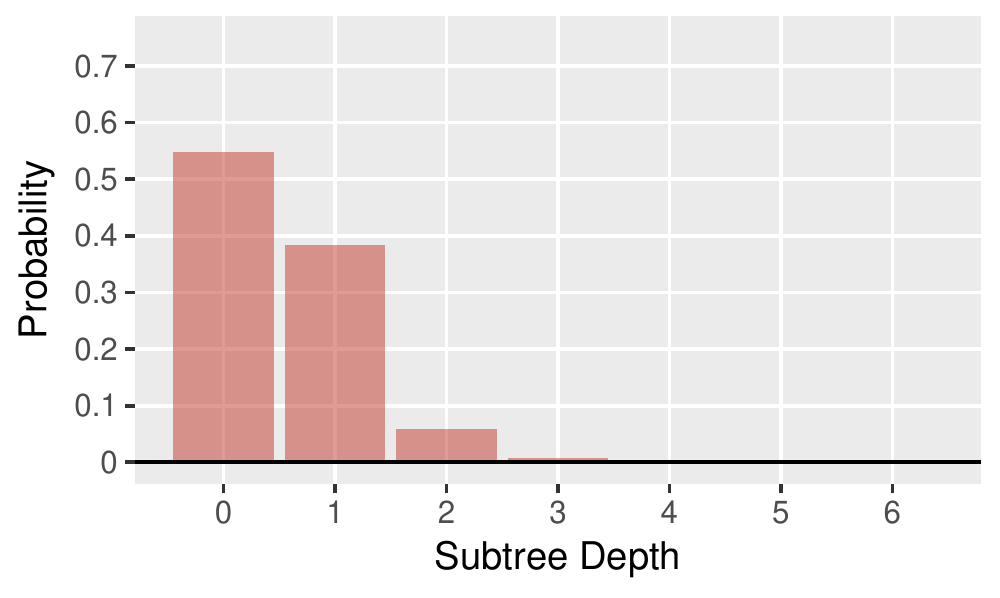}
                \caption{\texttt{D5}, Vanilla MCTS}
        \end{subfigure}
                \begin{subfigure}[b]{0.47\textwidth}
                \centering
                \includegraphics[width=\textwidth]{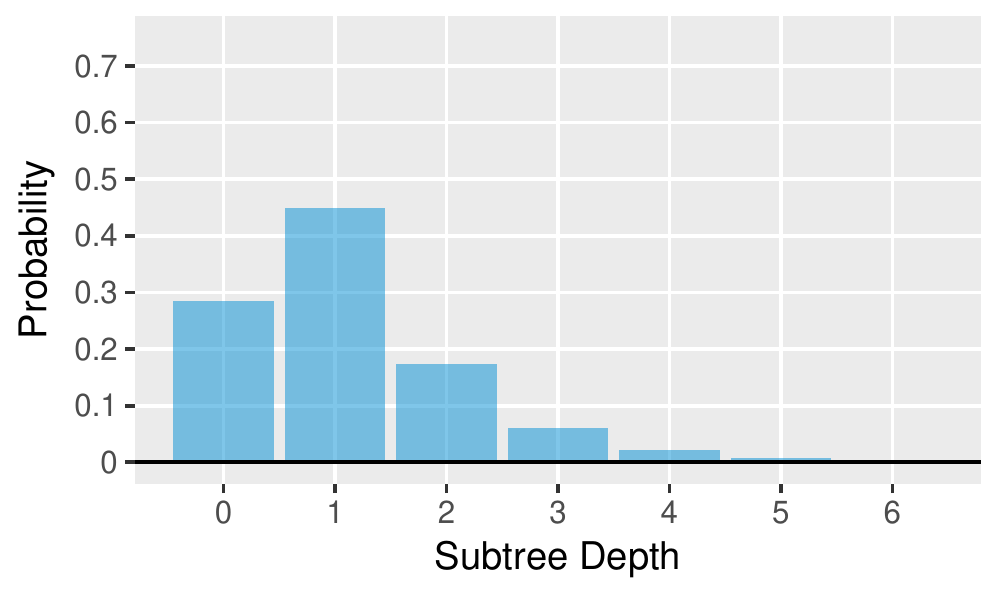}
                \caption{\texttt{D5}, Primal-Dual MCTS}
        \end{subfigure}\\

        \begin{subfigure}[b]{0.47\textwidth}
                \centering
                \includegraphics[width=\textwidth]{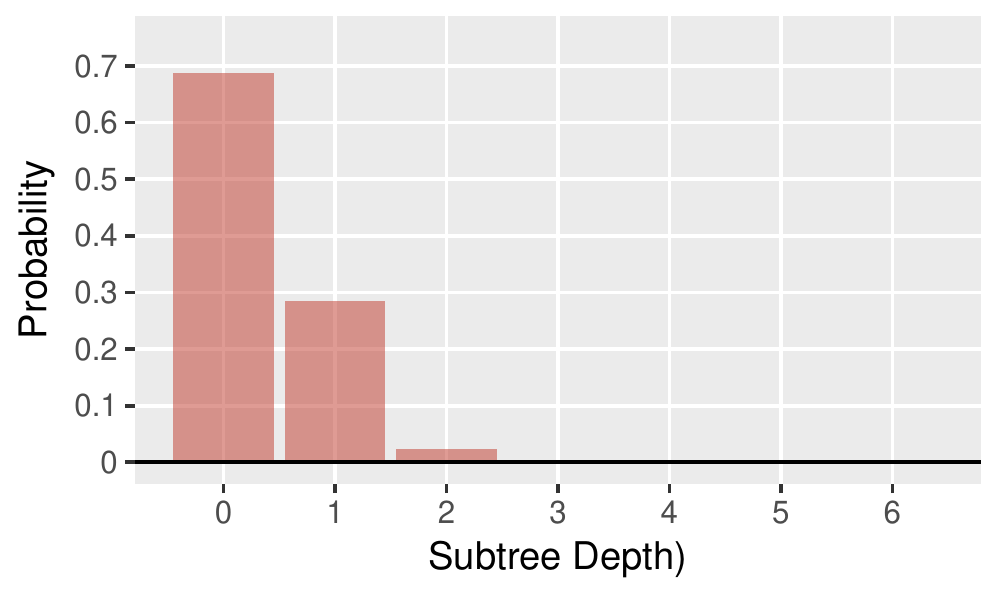}
                \caption{\texttt{D10}, Vanilla MCTS}
        \end{subfigure}
        \begin{subfigure}[b]{0.47\textwidth}
                \centering
                \includegraphics[width=\textwidth]{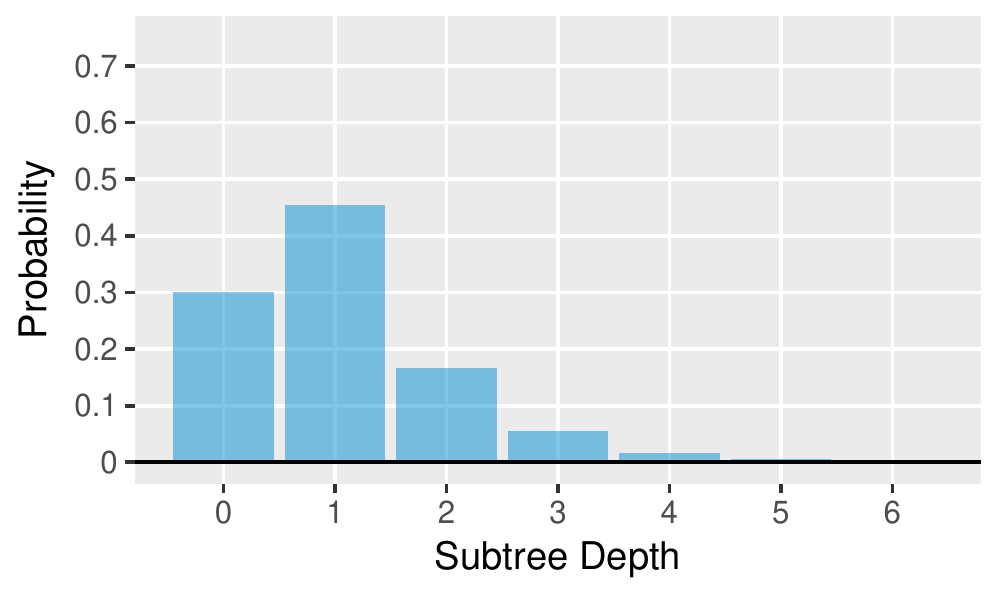}
                \caption{\texttt{D10}, Primal-Dual MCTS}
        \end{subfigure}\\

        \begin{subfigure}[b]{0.47\textwidth}
                \centering
                \includegraphics[width=\textwidth]{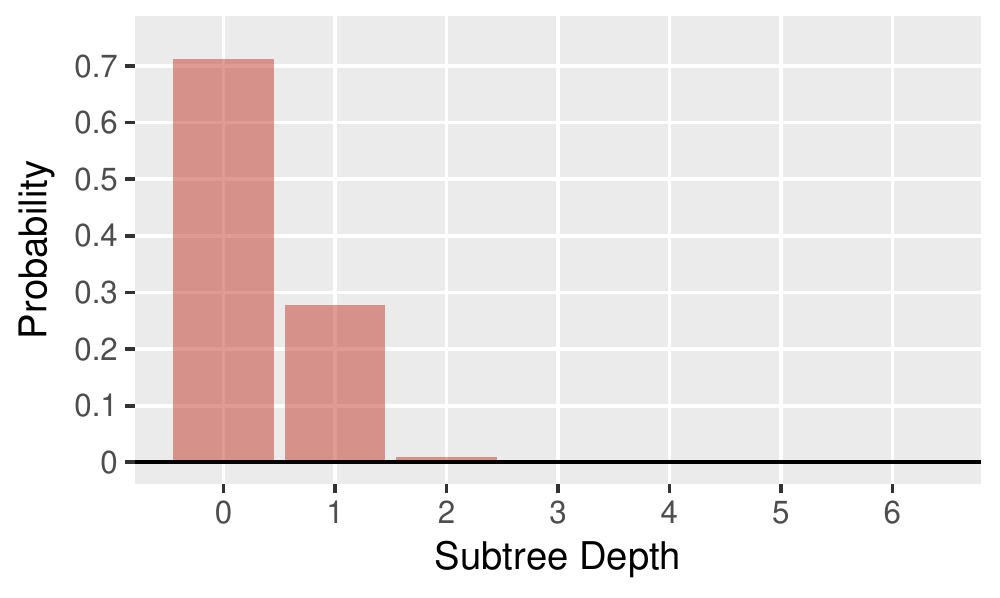}
                \caption{\texttt{D15}, Vanilla MCTS}
        \end{subfigure}       
        \begin{subfigure}[b]{0.47\textwidth}
                \centering
                \includegraphics[width=\textwidth]{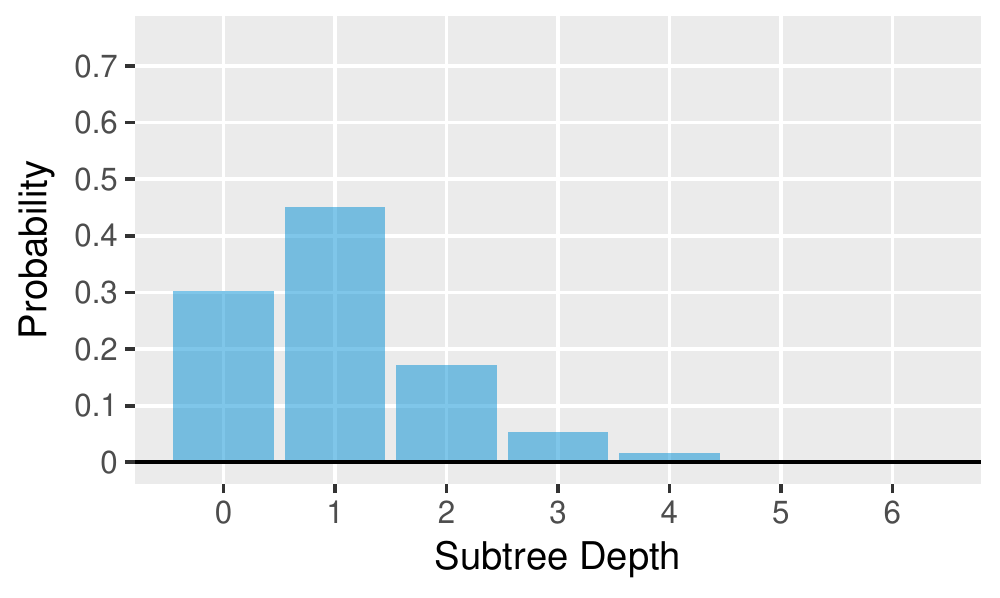}
                \caption{\texttt{D15}, Primal-Dual MCTS}
        \end{subfigure}
        \caption{Subtree Depth Distributions}
        \label{fig:subtree}
\end{figure}

To illustrate support for this intuition, we plot the empirical distributions of subtree depths (starting from any state node) for the two algorithms (after 25{,}000 iterations) on all three problem instances in Figure \ref{fig:subtree}. The \emph{depth} is defined in terms of state nodes and thus corresponds to decision epochs. We should note, however, that the physical size of the tree includes state-action nodes; thus, the actual depth is double what is shown in the Figure \ref{fig:subtree}. Although not visible from the figure, the depths of the full decision trees (i.e., from the root node) output by vanilla MCTS are 6, 5, and 5 for \texttt{D5}, \texttt{D10}, and \texttt{D15}, respectively. For Primal-Dual MCTS, the depths are 10, 10, and 9.

A major takeaway from Figure \ref{fig:subtree} is the relative insensitivity of Primal-Dual MCTS to an increasing action space. As the number of actions per state increases from 5 to 15, the empirical distributions stay roughly constant. Vanilla MCTS, on the other hand, is affected quite strongly: from instances \texttt{D5} to \texttt{D15}, the number of subtrees of zero depth increases from approximately 55\% to 70\%, signifying that many parts of the tree are expanded unnecessarily. Hence, these findings support the observations of \cite{Bertsimas2014} that vanilla MCTS suffers as the action branching factor increases. Our results also suggest that Primal-Dual MCTS is a viable alternative to remedy this issue. 

Table \ref{table:cputimes} provides some additional statistics: the CPU time per iteration (on a 2.6 GHz Intel Xeon E5-4650L processor) and the average number of expanded decisions per state. The trade-off for a more focused decision tree (i.e., more depth) within a fixed number of iterations is that information relaxation dual bound increases computation time by a factor of 4.5, on average. Furthermore, the number of expanded decisions per state is effectively halved by Primal-Dual MCTS.

\begin{table}[h]
\centering
\small
\begin{tabular}{@{}llllll@{}}\toprule
 CPU Time/Iter. (sec.) & Instance \texttt{D5}\quad  & Instance \texttt{D10}\quad & Instance \texttt{D15} \\
 \midrule
Vanilla MCTS  & 2.85 & 3.63  & 3.71  \\
Primal-Dual MCTS \quad \quad & 16.50  &	15.10  &	13.86  \\
\bottomrule
\end{tabular}

\vspace{10pt}

\begin{tabular}{@{}llllll@{}}\toprule
 Expansions/Node & Instance \texttt{D5}\quad  & Instance \texttt{D10}\quad & Instance \texttt{D15} \\
 \midrule
Vanilla MCTS  & 2.40	& 2.82 & 4.30 \\
Primal-Dual MCTS \quad \quad & 1.40 &	1.42 &	1.44 \\
\bottomrule
\end{tabular}
\caption{Per Iteration CPU Time (seconds) \& Average Number of Expanded Decisions per Node}
\label{table:cputimes}
\end{table}

Because of the problem size, we are unable to examine optimality of the decisions recommended by either of the algorithms. However, on the smallest problem instance \texttt{D5}, we observed that both Primal-Dual MCTS and vanilla MCTS \emph{converged to the same decision} (a ``consensus decision''). We are thus able to quantify, in these cases, the amount of improvement that information relaxation dual bounds can bring to vanilla MCTS. The improvement can be measured in terms of either the number of iterations or the amount of CPU time saved before the consensus decision is reached. 

Figure \ref{fig:lines} illustrates the behavior of the two algorithms at the root node for \texttt{D5}, where the red line indicates the value of the consensus decision and the blue lines indicate the values of the remaining decisions. Convergence occurs when the red line becomes the maximizer over all decisions. Notice that vanilla MCTS expands all five decisions and converges at iteration 65{,}511 while Primal-Dual MCTS only expands three decisions and converges at iteration 2{,}795. Therefore, taking into account the CPU time per iteration from Table \ref{table:cputimes}, we see that Primal-Dual MCTS requires $(2{,}795 \cdot 16.5)/(60{,}511 \cdot 2.85) \approx 27\%$ of the CPU time to discover the same solution found by vanilla MCTS. Thus, Figure \ref{fig:lines} suggests that Primal-Dual MCTS is more likely to recommend an optimal decision given a CPU budget, which is an important consideration in many applications. For example, it is often the case that in gameplay AI, a fixed amount of CPU usage is allocated per ``move.''

\begin{figure}[!ht]
\small
        \centering
        \begin{subfigure}[b]{0.47\textwidth}
                \centering
                \includegraphics[width=\textwidth]{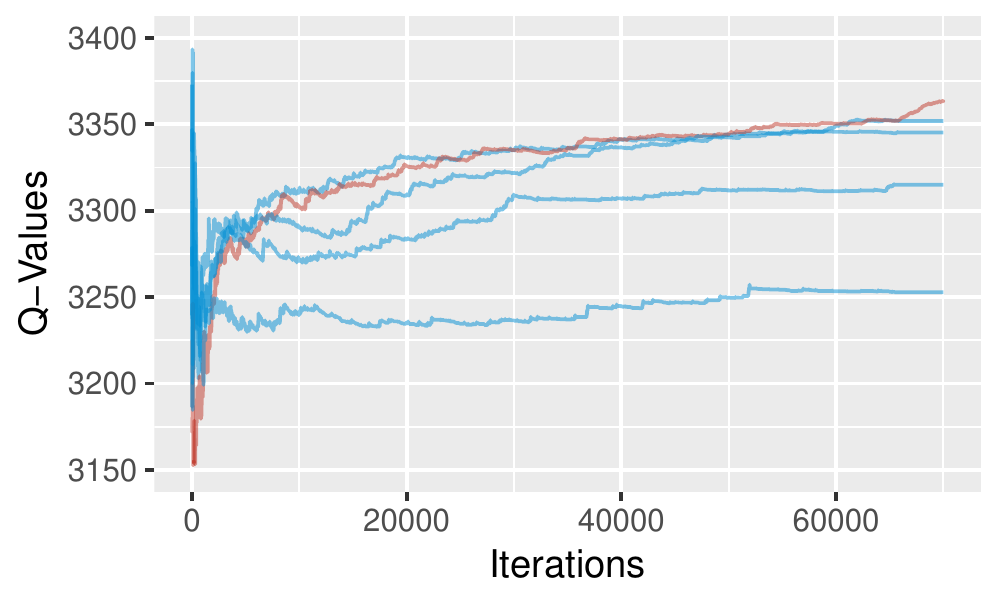}
                \caption{\texttt{D5}, Vanilla MCTS}
        \end{subfigure}
                \begin{subfigure}[b]{0.47\textwidth}
                \centering
                \includegraphics[width=\textwidth]{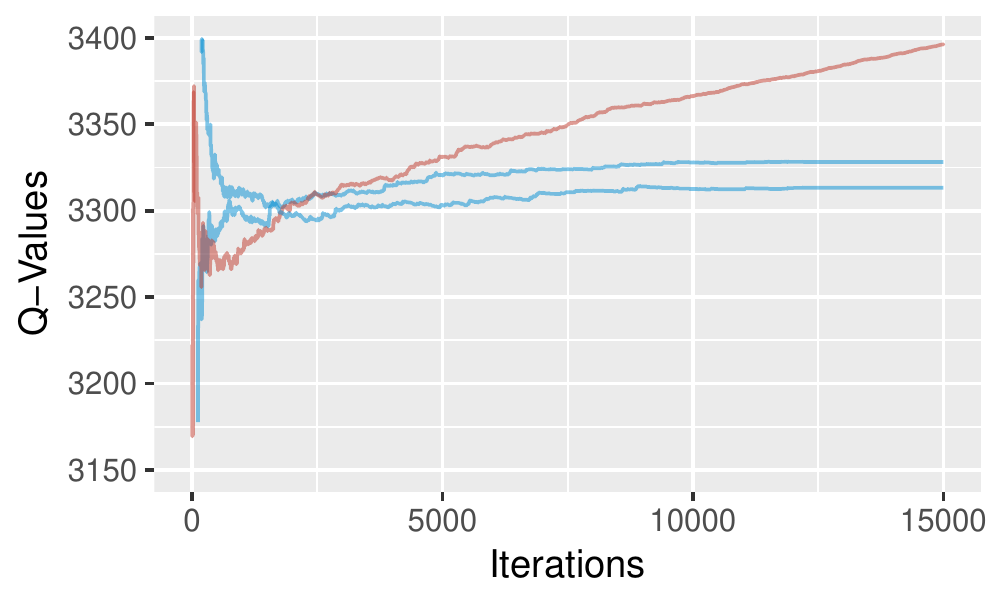}
                \caption{\texttt{D5}, Primal-Dual MCTS}
                \label{subfig:lines2}
        \end{subfigure}
        \caption{Convergence of $Q$-Values of State-Action Nodes}
        \label{fig:lines}
\end{figure}

To generate these plots and test for convergence, we observed 250{,}000 iterations of vanilla MCTS and 25{,}000 iterations of Primal-Dual MCTS. Every run of the algorithm on each problem instance thus requires \emph{5-8 days} of CPU time. Unfortunately, we are unable to run the larger problems, \texttt{D10} and \texttt{D15}, for enough iterations such that a single consensus decision is obtained by both vanilla MCTS and Primal-Dual MCTS. Thus, we cannot exactly quantify the benefits of using information relaxation bounds in \texttt{D10} and \texttt{D15}.

 However, we are able to give some insight into the behavior of Primal-Dual MCTS on these larger problems. Figure \ref{fig:linestenfifteen} shows that in both \texttt{D10} and \texttt{D15}, four decisions are expanded and convergence occurs. An interesting behavior of Primal-Dual MCTS is clearly displayed in Figure \ref{subfig:lines2}: due to noise, initial estimates of suboptimal decisions are high, which delays the expansion of the optimal decision until nearly iteration 500 when the noise is smoothed away. Once it is expanded, the optimal decision is quickly seen to be the best. Although we cannot directly compare with vanilla MCTS due to computational limitations, Figure \ref{fig:linestenfifteen} gives evidence of continued good behavior of Primal-MCTS on problems with larger action spaces.

\begin{figure}[!ht]
\small
        \centering
        \begin{subfigure}[b]{0.47\textwidth}
                \centering
                \includegraphics[width=\textwidth]{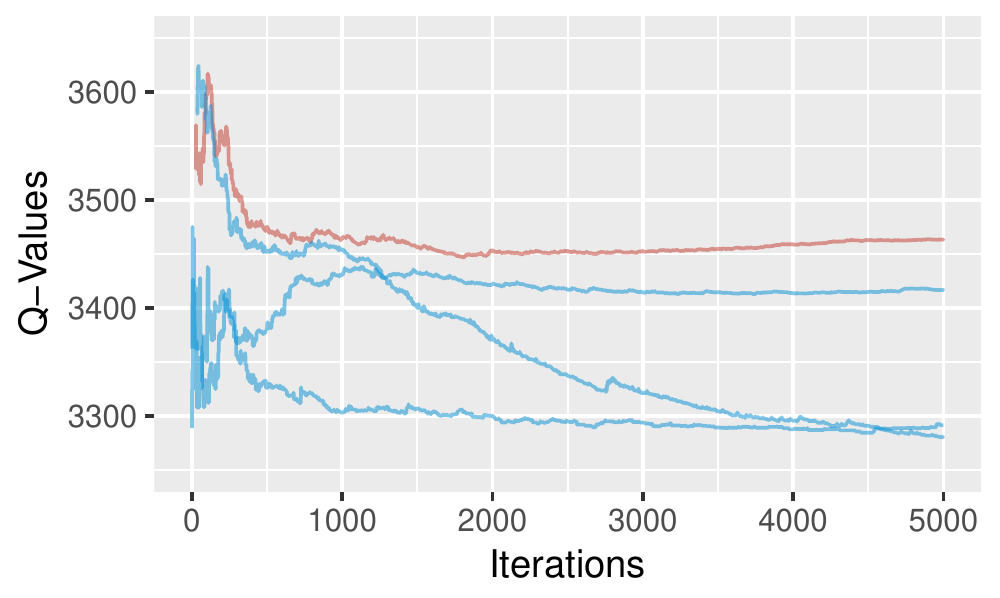}
                \caption{\texttt{D10}, Primal-Dual MCTS}
        \end{subfigure}
                \begin{subfigure}[b]{0.47\textwidth}
                \centering
                \includegraphics[width=\textwidth]{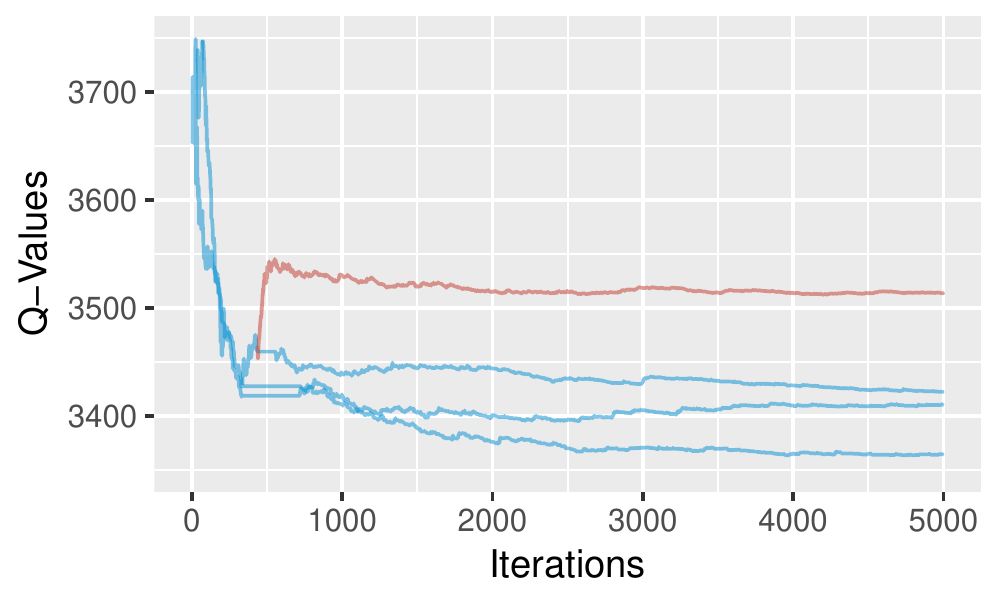}
                \caption{\texttt{D15}, Primal-Dual MCTS}
        \end{subfigure}
        \caption{Convergence of $Q$-Values of State-Action Nodes}
        \label{fig:linestenfifteen}
\end{figure}

\section{Conclusion}
\label{sec:conc}
In this paper, we study a new algorithm called Primal-Dual MCTS that attempts to address (1) the requirement that convergence of MCTS requires the full expansion of the decision tree and (2) the challenge of applying MCTS on problems with large action spaces \citep{Bertsimas2014}. To do so, we introduce the idea of using samples of an information relaxation upper bound to guide the tree expansion procedure of standard MCTS. It is shown that Primal-Dual MCTS converges to the optimal action at the root node, even if the entire tree is not fully expanded in the limit (an assumption typically required by MCTS convergence results). We also introduce the application of optimizing a single driver operating on a ride-sharing network subject to a real dataset of taxi rides occurring over a day in New Jersey. The empirical results indicate that Primal-Dual MCTS can significantly reduce the computation time needed to converge to an optimal root node decision when compared to vanilla MCTS.

\clearpage
\appendix

\section{Proofs}
\label{sec:proofs}
\lemio*
\begin{proof}
 These results can be argued by moving down the tree. First, $\x_0 = s_0$ is visited infinitely often by default. Assume the lemma holds for all nodes down to all state nodes at time $t$. Let $\y_{t} = (\x_t, a_t) \in \mathcal Y_t$ where $\P(\y_t \in \mathcal Y^\infty)>0$. Clearly, $\{\x_t \in \mathcal X^\infty\} \subseteq \{\y_t \in \mathcal Y^\infty\}$. Define $\eta^i(\x_t)$ to be the random variable representing the $i$-th iteration for which that $\x_t$ is visited. By the induction hypothesis, it must be the case that on $\{\y_t \in \mathcal Y^\infty\}$, the sequence $(\eta^1(\x_t), \eta^2(\x_t), \eta^3(\x_t), \ldots)$ is a subsequence of $(1, 2, 3, \ldots)$ that goes to $\infty$. Hence, it follows by Assumption \ref{ass:one}(i) that
 \[
\sum_{i = 1}^\infty \P( \pis(\x_t, \mathscr T^{\eta^i(\x_t)-1}) =  \y_t \,  | \, \y_t \in \mathcal Y^\infty) = \infty.
 \]
Applying the Borel-Cantelli lemma \citep{Breiman1992} with the probability measure $\P( \; \cdot \;  | \, \y_t \in \mathcal Y^\infty)$, we can conclude that $v^n(\y_t) \rightarrow \infty$ almost everywhere on the event that $\y_t$ is expanded, $\{\y_t \in \mathcal Y^\infty\}$.

Next, let $\y_t' = (\x_t',a_t') \in \mathcal Y_t$. A similar argument using Assumption \ref{ass:one}(iii) and another application of the Borel-Cantelli lemma implies that $l^n(\y_t') \rightarrow \infty$ almost everywhere on the event $\{\y_t' \in \tilde{\mathcal Y}^\infty\}$. We use the fact that if $\y_t' \in \tilde{\mathcal Y}^\infty$ then $\x_t' \in \mathcal X^\infty$ and by the induction hypothesis, $\x_t'$ is visited infinitely often. Thus, the lemma holds down to arbitrary state-action nodes $\y_t \in \mathcal Y_t$.

Now, by Case 2 of the expansion step, we know that only states that can be reached with positive probability are added. Therefore, a similar argument as the above along with Assumption \ref{ass:one}(ii) and the Borel-Cantelli lemma show that almost everywhere on $\{\x_{t+1} \in \mathcal X^\infty\}$, the node $\x_{t+1}$ is visited infinitely often. By induction, the proof is complete.
\end{proof}

\lemargmax*
\begin{proof}
Consider any $\omega$ in the event defined in the statement of the lemma and fix this $\omega$ throughout this proof. Suppose for this sample path $\omega$ that there is an optimal state-action node that is never expanded, i.e., $\y_t^* = (\x_t,a_t^*) \in \tilde{\mathcal Y}^\infty$.
By Lemma \ref{lem:io}, we know that $l^n(\y_t^*) \rightarrow \infty$ and let us denote the set of iterations for which $\y_t^* \in \Atilde^n(\x_t)$ (i.e., is a candidate action) by $\mathcal N_x(\x_t,\y_t^*) \subseteq \mathcal N_x$. Then, since this optimal action was never expanded, it must hold true that for any iteration $n \in \mathcal N_x(\x_t,\y_t^*)$, the dual bound approximation does not exceed the current value approximation (i.e., ``current best''):
\[
\bar{u}^n(\y_t^*) \le (1-\lambda^{n-1}) \, \tilde{V}^{n-1}(\x_t) + \lambda^{n-1} \max\nolimits_{\y_t \in \mathcal Y^{n-1}(\x_t)} \bar{Q}^{n-1}(\y_t).
\]
Since $|\mathcal N_x(\x_t,\y_t^*)| = \infty$, we may pass to the limit and utilize conditions (i) and (ii) in the statement of the lemma along with Assumption \ref{ass:one}(v) to obtain $u^{\nu}(\y_t^*) \le \max_{\y \in \mathcal Y^\infty(\x_t)} Q^*(\y).$ Because we have made the assumption that $\mathcal Y^\infty(\x_t)$ does not contain an optimal action,
\begin{equation}
u^{\nu}(\y_t^*) \le   \max\nolimits_{\y_t \in \mathcal Y^\infty(\x_t)} Q^*(\y_t) < V^*(\x_t).
\label{eq:contra1}
\end{equation}
On the other hand, applying Proposition \ref{prop:upperbound} to the optimal state-action pair $\y_t^*$, we see that
\[
u^{\nu}(\y_t^*) \ge Q^*(\y_t^*) = V^*(\x_t),
\]
a contradiction with (\ref{eq:contra1}). Thus, we conclude that our original assumption was incorrect and it must be the case an optimal state-action node is expanded in the limit for our fixed $\omega$. Since the sample path $\omega$ was arbitrary, the conclusion holds for any $\omega \in A$.
\end{proof}

\mainthm*

\begin{proof}
Before moving on to the main proof, we state a result that will be useful later on. Let $\y_t \in \mathcal Y$. For any $k_0 \in \mathbb N$, let $\alpha_{k_0}^{k}(\y_t)$ be a stepsize where 
\[
\sum\nolimits_{k=k_0+1}^\infty \alpha_{k_0}^{k}(\y_t) = \infty \quad \text{and} \quad \sum\nolimits_{k=k_0+1}^\infty (\alpha_{k_0}^{k}(\y_t))^2 < \infty \quad a.s.
\]
For any initial value $\bar{Q}_{k_0}^{k_0}(\y_t)$, consider the iteration
\begin{equation}
\bar{Q}_{k_0}^{k+1}(\y_t) = \bar{Q}_{k_0}^k(\y_t) + \alpha_{k_0}^{k+1}(\y_t) \, \bigl[\bar{V}^k((\y_t,S_{t+1})) - \bar{Q}^k_{k_0}(\y_t) \bigr] \; \; \text{for} \; k \ge k_0,
\label{eq:auxiter}
\end{equation}
where $S_{t+1} \sim \P(\; \cdot \; | \, \y_t)$ and $\bar{V}^k(\y) \rightarrow V^*(\y)$ almost surely for every $\y$. We can write 
\begin{align*}
\bar{V}^k&((\y_t,S_{t+1})) - \bar{Q}^k_{k_0}(\y_t) \\
&= Q^*(\y_t)  - \bar{Q}^k_{k_0}(\y_t) + V^*((\y_t,S_{t+1})) - Q^*(\y_t) + \bar{V}^k((\y_t,S_{t+1})) - V^*((\y_t,S_{t+1})).
\end{align*}
Since $\bar{V}^k((\y_t,S_{t+1})) - V^*((\y_t,S_{t+1})) \rightarrow 0$ almost surely and $\E[V^*((\y_t,S_{t+1})) - Q^*(\y_t) \, | \, \y_t] = 0$, it follows by \cite[Theorem 2.4]{Kushner2003} that
\begin{equation}
\bar{Q}_{k_0}^k(\y_t) \rightarrow Q^*(\y_t) \quad a.s.
\label{eq:auxconv}
\end{equation}
The conditions of \cite[Theorem 2.4]{Kushner2003} are not difficult to check given our setting with bounded contributions.

With this auxiliary result in mind, we move on to the proof of Theorem \ref{thm:main}. We aim to show via induction that for any $\x_{t+1} \in \mathcal X_{t+1}$,
\begin{align}
\bar{V}^n(\x_{t+1}) &\rightarrow V^*(\x_{t+1}) \, \indicate{\x_{t+1} \in \mathcal X_{t+1}^\infty} \quad a.s.,\label{eq:indconv1}
\end{align}
This means that convergence to the optimal value function occurs on the event that the node is expanded. Otherwise, the value function approximation stays at its initial value of zero. We proceed by backwards induction from $t+1=T$. The base case is clear by (\ref{eq:simobs}) and because $V^*(\x_T) =0$ for $\x_T \in \mathcal X_T$. The induction hypothesis is that (\ref{eq:indconv1}) holds for an arbitrary $t+1$.

In order to complete the inductive step, our first goal is to show that the state-action value function converges on the event that the node $\y_t = (\x_t,a_t) \in \mathcal Y_t$ is expanded:
\begin{align}
\bar{Q}^n(\y_{t}) \rightarrow Q^*(\y_{t}) \, \indicate{\y_{t} \in \mathcal Y_{t}^\infty} \quad a.s.,\label{eq:indconv2}
\end{align}
We will give an $\omega$-wise argument. Let us fix an $\omega \in \Omega$. If $\y_t \not \in \mathcal Y_t^\infty$, then $\bar{Q}^n(\y_t)$ is never updated and thus converges to zero. Now suppose that the node is expanded, i.e., $\y_t \in \mathcal Y_t^\infty$. Let $\alpha_v^n(\y_t) = 1/v^n(\y_{t}) \, \mathbf{1}_{\{\y_t \in \, \x_{\tau_s}^n\}}$, where the notation $\y \in \x_{\tau_s}^n$ is used to indicate that $\y$ is visited on iteration $n$. Thus, we can rewrite (\ref{eq:backprop1}) and (\ref{eq:backprop4}) in the form of a stochastic approximation step as follows:
\begin{equation}
\bar{Q}^n(\y_t) = \bar{Q}^{n-1}(\y_t) + \alpha_v^n(\y_t) \bigl[\bar{V}^n(\pis(\y_t,\mathscr T^{n-1})) - \bar{Q}^n(\y_t) \bigr].
\label{eq:stochapprox}
\end{equation}
We will analyze the tail of this iteration. Since $v^n(\y_t) \rightarrow \infty$ by Lemma \ref{lem:io} and $|\mathcal S|$ is finite, it is clear that there exists an iteration $N^*$ (depends on $\omega$) after which all state nodes $\x_{t+1} = (\y_t,s_{t+1})$ where $s_{t+1}$ is reachable with positive probability are expanded. By Assumption \ref{ass:one}(iii), the tail of the iteration (\ref{eq:stochapprox}) starting at $N^*$ is equivalent to
\[
\bar{Q}_{N^*}^{n+1}(\y_t) = \bar{Q}_{N^*}^n(\y_t) + \alpha_v^{n+1}(\y_t) \, \bigl[\bar{V}^n((\y_t,S_{t+1})) - \bar{Q}^n_{N^*}(\y_t) \bigr] \; \; \text{for} \; n \ge N^*,
\]
where $\bar{Q}_{N^*}^{N^*}(\y_t) = \bar{Q}^{N^*}(\y_t)$. Define $\eta^i(\y_t)$ be the $i$-th iteration for which that $\y_t$ is visited. By Lemma \ref{lem:io}, we know that $(\eta^1(\y_t), \eta^2(\y_t), \eta^3(\y_t), \ldots)$ is a subsequence of $(1, 2, 3, \ldots)$ that goes to $\infty$. Let $i^*$ be the smallest $i$ such that $ \eta^i(\y_t) > N^*$. Hence,
\[
\sum\nolimits_{n=N^*+1}^\infty \alpha_v^{n}(\y_t) = \sum\nolimits_{i = i^*}^\infty \alpha_v^{\eta^i(\y_t)}(\y_t) = \sum\nolimits_{i = i^*}^\infty 1/i = \infty.
\]
Similarly, $\sum\nolimits_{n=N^*+1}^\infty (\alpha_v^{n}(\y_t))^2 < \infty$. Since $(\y_t,S_{t+1}) \in \mathcal X_{t+1}^\infty$ for every realization of $S_{t+1}$, by the induction hypothesis, it holds that $\bar{V}^n(\y_t,S_{t+1}) \rightarrow V^*(\y_t,S_{t+1})$. Thus, by the auxiliary result stated above in (\ref{eq:auxconv}), it follows that
$\bar{Q}_{N^*}^{n}(\y_t) \rightarrow Q^*(\y_t)$ and so we can conclude that
\begin{equation}
\bar{Q}^n(\y_t) \rightarrow Q^*(y_t)
\end{equation}
for when our choice of $\omega$ is in $\{\y_t \in \mathcal Y_t^\infty\}$, which proves (\ref{eq:indconv2}).

The next step is to examine the dual upper bounds. Analogously, we would like to show:
\begin{align}
\bar{u}^n(\y_{t}) \rightarrow u^{\nu}(\y_t) \, \indicate{\y_{t} \in \tilde{\mathcal Y}^\infty_t} \quad a.s.\label{eq:indconv3}
\end{align}
The averaging iteration (\ref{eq:duallookahead}) can be written as
\[
\bar{u}^n(\y_t' ) = \bar{u}^{n-1}(\y_t') - \alpha^n(\y_t') \bigl[ \bar{u}^{n-1}(\y_t' ) -  \hat{u}^n(\y_t' ) \bigr],
\]
where $\E[\hat{u}^n(\y_t')] = u^{ \nu}(\y_t')$. There is no bias term here. Under Lemma \ref{lem:io}, Assumption \ref{ass:one}(iv), and our finite model and bounded contributions, an analysis similar to the case of (\ref{eq:indconv2}) allows us to conclude (\ref{eq:indconv3}).

Finally, we move on to completing the inductive step, i.e., (\ref{eq:indconv1}) with $t$ replacing $t+1$. Consider $\x_t \in \mathcal X_t$ and again let us fix an $\omega \in \Omega$. If $\x_t \not \in \mathcal X_t^\infty$, then there are no updates and we are done, so consider the case where $\x_t \in \mathcal X_t^\infty$. The conditions of Lemma \ref{lem:argmax} are verified by (\ref{eq:indconv2}) and (\ref{eq:indconv3}), so the lemma implies that an optimal action $a_t^* \in \argmax_{a \in \mathcal A} Q^*(\x_t,a)$ is expanded. Consequently, it follows by (\ref{eq:indconv2}) that 
\[
\max\nolimits_{\y_t \in \mathcal Y^\infty(\x_t)} \bar{Q}^n(\y_t) \rightarrow Q^*(\y_t^*) = V^*(\x_t).
\]
By Assumption \ref{ass:one}(v) and the backpropagation update (\ref{eq:backprop2}), we see that $\bar{V}^n(\x_t) \rightarrow V^*(\x_t)$, which proves (\ref{eq:indconv1}). We have shown that all value function approximations converge appropriately for expanded nodes in random limiting tree $\mathscr T^\infty$.

Now we move on to the second part of the theorem, which concerns the action taken at the root node. If the optimal action is unique (i.e., there is separation between the best action and the second best action), then (\ref{eq:indconv1}) allows us to conclude that the limit of the set of maximizing actions $\argmax_{\y \in \mathcal Y^n(\x_0)} \bar{Q}^n(\y)$ is equal to $\argmax_{\y_0 = (\x_0,a)} Q^*(\y_0)$. However, if uniqueness is not assumed (which we have not), then we may conclude that each accumulation point of $\argmax_{\y \in \mathcal Y^n(\x_0)} \bar{Q}^n(\y)$ is an optimal action at the root node:
\[
\limsup_{n \rightarrow \infty} \, \argmax_{\y \in \mathcal Y^n(\x_0)} \bar{Q}^n(\y) \subseteq \argmax_{\y_0 = (\x_0,a)} Q^*(\y_0) \; \; a.s.
\]
The proof is complete.
\end{proof}



\clearpage
\bibliographystyle{plainnat}
\bibliography{/Users/drjiang/Documents/Dropbox/Pittsburgh/Bibtex/Bib,/Users/drjiang/Documents/Dropbox/Pittsburgh/Bibtex/Risk,/Users/drjiang/Documents/Dropbox/Pittsburgh/Bibtex/EV,/Users/drjiang/Documents/Dropbox/Pittsburgh/Bibtex/MCTS}

\end{document}